\documentclass{birkjour}
\usepackage{amsmath}
\usepackage{amssymb}
\usepackage{amsthm}

\usepackage{verbatim}

\usepackage[utf8]{inputenc}

\usepackage{graphicx}
\usepackage{color}
\usepackage{enumerate}

\usepackage{esint}
\usepackage[all]{xy}
\usepackage{framed}
\usepackage{bbm}
\usepackage{subcaption}

\theoremstyle{plain}
\newtheorem{theorem}{Theorem}[section]
\newtheorem{corollary}[theorem]{Corollary}
\newtheorem{lemma}{Lemma}[section]

\newtheorem{proposition}{Proposition}[section]

\theoremstyle{definition}
\newtheorem{definition}{Definition}[section]

\newtheorem{example}{Example}[section]

\numberwithin{equation}{section}

\newcommand{\bfa}{\mathbf{a}}

\newcommand{\bfq}{\mathbf{q}}

\newcommand{\D}{\mathcal{D}}
\newcommand{\Sn}{\mathcal{S}^n}

\newcommand{\cH}{\mathcal{H}}

\newcommand{\dd}{\mathrm{\,d}}

\newcommand{\mR}{\mathbb{R}}
\newcommand{\oO}{\overline{\Omega}}
\newcommand{\oT}{\overline{\Theta}}
\newcommand{\oB}{\overline{B}}
\newcommand{\oR}{\overline{\mR}}
\newcommand{\oF}{\overline{F}}
\newcommand{\tF}{\widetilde{F}}

\DeclareMathOperator{\tr}{tr}

\DeclareMathOperator{\di}{div}
\DeclareMathOperator{\sgn}{sgn}

\DeclareMathOperator{\dist}{dist}

\DeclareMathOperator{\ex}{ex}

\begin{document}

\author[K. K. Brustad]{Karl K. Brustad}
\title[Comparison principle for second order elliptic equations]{On the comparison principle for\\ second order elliptic equations without\\ first and zeroth order terms}
\address{Frostavegen 1691\\ 7633 Frosta\\ Norway}
\email{brustadkarl@gmail.com}

\begin{abstract}
We consider the comparison principle for semicontinuous viscosity sub- and supersolutions of second order elliptic equations on the form $F(\cH w,x) = 0$. A structural condition on the operator is presented that seems to unify the different existing theories. A new result is obtained and the proofs of the classical results are simplified.
\end{abstract}

\maketitle

\section{Introduction}

When the concept of viscosity solutions was developed in the 1980s, it was initially not known whether a solution of the Dirichlet problem for a second order equation $F(\cH w,\nabla w,w,x) = 0$ would be unique, except for in a few particular cases. The situation was resolved by a regularization procedure introduced by Jensen, which eventually led to the very efficient, but somewhat mysterious-looking \emph{structure condition} (3.14) in \cite{MR1118699} on the operator $F$. In our setting, with equations
\begin{equation}\label{eq:equation}
F(\cH w,x) = 0
\end{equation}
independent of $\nabla w$ and $w$, the condition reads as follows.
\begin{equation}\label{eq:classical_structure_cond}
\begin{aligned}
&\text{There is a modulus of continuity $\omega$ such that}\\
&F(X,x) - F(Y,y) \leq \omega\Big(\alpha|x-y|^2 + |x-y| \Big)\\
&\text{whenever $x,y\in\Omega$, $\alpha>0$, and $X,Y\in\Sn$ satisfying}\\
&-\alpha\begin{bmatrix}
I & 0\\ 0 & I
\end{bmatrix} \leq \begin{bmatrix}
X & 0\\ 0 & -Y
\end{bmatrix} \leq \alpha  \begin{bmatrix}
I & -I\\ -I & I
\end{bmatrix}.
\end{aligned}
\end{equation}

The \emph{comparison principle}, by which the uniqueness of solutions is immediate, could then be proved for a large class of equations.

\begin{definition}[Comparison principle]
If $v\in USC(\oO)$ is a viscosity subsolution and $u\in LSC(\oO)$ is a viscosity supersolution of \eqref{eq:equation} in an open and bounded set $\Omega\subseteq\mR^n$ with $v\rvert_{\partial\Omega}\leq u\rvert_{\partial\Omega}$, then $v\leq u$ also in $\Omega$.
\end{definition}


A completely different approach was taken in a program initiated by Krylov in \cite{MR1284912}, further developed by e.g. Harvey, Lawson, Cirant, and Payne.
Krylov came up with a simple, but powerful trick: replace a badly behaving equation $F(\cH w) = 0$ with the nicer equivalent $\oF(\cH w) = 0$ where $\oF$ is the signed distance function from the null-level set of $F$ in $\Sn $. In this way he proved the comparison principle for convex elliptic equations by reducing them to Bellman equations, which was exactly the kind of equations manageable by the viscosity theory \emph{before} Jensen's regularization technique. The convexity of $F$ is in fact unessential (\cite{MR2487853}), and in \cite{MR3677871} the method is extended to also include the non-autonomous case $F(\cH w,x) = 0$. Now, the \emph{sub- and superlevel sets}
\[\Theta_-(x) := \left\{X\in \Sn \;|\; F(X,x)\leq 0\right\},\quad \Theta_+(x) := \left\{X\in \Sn \;|\; F(X,x)\geq 0\right\},\]
depend on $x$ and the comparison principle was proved under a \emph{uniform Hausdorff continuity} of $\Theta_+(x)$. i.e.,

\begin{equation}\label{eq:Hausdorff_cond}
\lim_{y\to x}\sup_{X\in \Sn }\left|\dist(X,\Theta_+(x)) - \dist(X,\Theta_+(y))\right| = 0\qquad\text{uniformly in $\Omega$.}
\end{equation}
Actually, \cite{MR3677871} study the Dirichlet problem for \emph{differential inclusions} $\cH w\in\partial\Theta(x)$ where $\Theta$ is a given proper and closed positive \emph{elliptic map} (\eqref{eq:elliptic_set_def}). The notion of an operator $F$ is abandoned. However, the suggested weak interpretation is shown to be equivalent to viscosity when $\Theta(x)$, is the superlevel set of a continuous elliptic operator.

Needless to say, the two conditions \eqref{eq:classical_structure_cond} and \eqref{eq:Hausdorff_cond} are very different and indeed, none is weaker than the other.

The biggest problem with \eqref{eq:Hausdorff_cond} is that it does not take full advantage of the \emph{Theorem of Sums} (Lemma \ref{lem:ishii}). Thus, it cannot handle operators with 'diverging null-level sets' in $\Sn$ (as conceptually illustrated in Figure \ref{fig:bad}). In particular, \eqref{eq:Hausdorff_cond} does not hold for linear equations $\tr(A(x)\cH w) = f(x)$ with non-constant coefficients $A\colon\Omega\to\Sn_+$.

On the other hand, \eqref{eq:classical_structure_cond} have problems that are not present in \eqref{eq:Hausdorff_cond}:
Firstly, Example \ref{ex:determinat_eq_not_classical} and Example \ref{ex:canonical_determinat} show that \eqref{eq:classical_structure_cond} is not \emph{invariant under equivalence}. Meaning that, given two equivalent equations $F(\cH w,x)=0$ and $\oF(\cH w,x)=0$ (i.e., equations having the exact same set of sub- and supersolution) then \eqref{eq:classical_structure_cond} may hold for $\oF$ but not for $F$. An immediate conclusion is of course that \eqref{eq:classical_structure_cond} is not necessary. Secondly, \eqref{eq:classical_structure_cond} is insufficient by its own. Some kind of additional \emph{non-degeneracy} in the operator has to be assumed. For this reason, it is not possible to compare \eqref{eq:classical_structure_cond} with other results directly. Equations on the form $F(\cH w) = f(x)$ are especially troublesome since \eqref{eq:classical_structure_cond} then automatically holds for every (uniformly) continuous $f$, thereby giving no clue on what the sufficient non-degeneracy of $F$ must be. In our view, this is the major drawback of the classical structure condition. A non-degeneracy like \emph{strict ellipticity} or the weaker $F(X+tI,x)-F(X,x) \geq \ell(t)>0$, $t>0$, would do, but even the latter is not necessary as illustrated by Proposition \ref{prop:semiautonomous} and Proposition \ref{prop:alog_equation}.
Finally,
it should be mentioned that continuity of $F$ is not something that has to be assumed. Also, in \cite{MR2487853} and \cite{MR3677871} a more natural notion of ellipticity \eqref{eq:elliptic_set_def} is used, which is weaker than the classical (degenerate) ellipticity \eqref{eq:elliptic}.



The purpose of this paper should then be clear: To present a condition on $F$ that incorporates both \eqref{eq:classical_structure_cond} and \eqref{eq:Hausdorff_cond}, and that is without the problems listed above.
Before we give the precise statement of the comparison result, we shall try to describe what our condition \eqref{eq:cond} in Theorem \ref{thm} below means. 
For a parameter $\delta\geq 0$ we introduce the \emph{perturbed sub- and superlevel sets}
\begin{equation}\label{eq:Tddef}
\begin{aligned}
\Theta^\delta_-(x) &:= \Big\{X(I+\delta X)^{-1}\;\;\big|\;\; F(X,x)\leq 0, \; \delta|X| < 1\Big\},\\
\Theta^\delta_+(x) &:= \Big\{X(I-\delta X)^{-1}\;\;\big|\;\; F(X,x)\geq 0, \; \delta|X| < 1\Big\}.
\end{aligned}
\end{equation}
Note that $\Theta_\pm^0(x) = \Theta_\pm(x)$. The expression $X(I-\delta X)^{-1}$ is familiar to those who are acquainted with the classical structure condition. It is related to the matrix inequality in \eqref{eq:classical_structure_cond}. In fact, as mentioned in e.g. the remarks following Theorem IV.1 in \cite{MR1031377}, it is enough to consider $Y=X(I-\frac{1}{\alpha}X)^{-1}$ in order to verify \eqref{eq:classical_structure_cond}. See Proposition \ref{prop:matrixineq}, and in general Section \ref{sec:Tdelta} for an account on the mapping
\[T_\delta(X) := X(I-\delta X)^{-1}.\] Most importantly, $T_\delta(X)\geq X$, which by ellipticity implies $\Theta_+^\delta(x) \subseteq \Theta_+(x)$. However, at two different points $x$ and $y$, it is possible that $\Theta_+(y)$ no longer contains $\Theta_+^\delta(x)$. The 'lack of containment' can be measured by the \emph{excess} of $\Theta_+^\delta(x)$ over $\Theta_+(y)$. It is defined as the smallest number $\tau\geq 0$ so that $\Theta_+^\delta(x)$ is a subset of the $\tau$-fattening of $\Theta_+(y)$. We shall require that \emph{The excess vanish as $\delta\to 0$ whenever $x,y\to x_0$ such that $|x-y|^2/\delta \to 0$.}
Our condition is therefore a certain continuity of the set-valued function $(\delta,x)\mapsto\Theta_+^\delta(x)$ at $\delta = 0$ and $x=x_0$.

\begin{theorem}\label{thm}
Let $F\colon \Sn \times\Omega \to \oR$ be elliptic at 0 (Def. \ref{def:0}).
If, for all $x_0\in\Omega$, $\partial\Theta_+(x_0)= \partial\Theta_-(x_0)\neq \emptyset$ and
\begin{equation}
\lim_{t\to 0^+}\sup_{x,y\in B_{t}(x_0)} \ex\left( \Theta^{\frac{|x-y|^2}{t}}_{+}(x), \Theta_{+}(y) \right) = 0,
\label{eq:cond}
\end{equation}
then
the comparison principle holds for the equation $F(\cH w,x) = 0$ in the open and bounded $\Omega\subseteq\mR^n$.
\end{theorem}

Since the excess is
$\ex(\Theta_1,\Theta_2) = \sup_{Z\in\Theta_1}\dist(Z,\Theta_2)$,
where, as always
\begin{equation}
\dist(Z,\Theta) := \inf_{Y\in\Theta}|Z-Y|,
\label{eq:distdef}
\end{equation}
the left-hand side of \eqref{eq:cond} is equal to
\begin{align*}
&\lim_{t\to 0^+}\sup_{\substack{x,y\in B_t(x_0)\\ Z\in\Theta^\delta_+(x)}} \dist\left( Z, \Theta_+(y) \right),\\
&\lim_{t\to 0^+}\sup_{\substack{x,y\in B_t(x_0)\\ X\in\Theta_+(x)\cap B_{1/\delta}}} \dist\left( T_\delta(X), \Theta_+(y) \right),\qquad\text{and}\\
&\lim_{t\to 0^+}\sup_{\substack{x,y\in B_t(x_0)\\ X\in\Theta_+(x)\cap B_{1/\delta}}}\inf_{Y\in\Theta_+(y)} \left| T_\delta(X) - Y\right|,
\end{align*}
where $\delta:=|x-y|^2/t$. We shall be working mostly with the first and second formulation.

Our non-degeneracy condition $\partial\Theta_+(x_0)= \partial\Theta_-(x_0)\neq \emptyset$ is the same as the one in \cite{MR3677871}, albeit in a different form. The non-emptiness is not necessary but it simplifies the exposition and it rules out some pathological cases (Consider for example the equation $F(\cH w,x) := 1 = 0$ where the comparison principle vacuously holds).
The requirement $\partial\Theta_+(x_0)= \partial\Theta_-(x_0)$ corresponds to $F(X_0 -tI,x_0)<0<F(X_0+tI,x_0)$ whenever $F(X_0,x_0)=0$ and $t>0$ (Proposition \ref{thm:2}, (iv) $\iff$ (vi)). We do not know if it is sharp,\footnote{unless $x\mapsto\dist(X_0,\Theta_{+}(x))$ is continuous at $x_0$. Then a counterexample can be produced in the same way as in the proof (i) $\Rightarrow$ (ii) of Proposition \ref{thm:2}.} but it is weaker than any other non-degeneracy condition we have been able to find in the literature.

\begin{figure}[h]
	\centering
	\begin{subfigure}[b]{0.45\textwidth}
		\centering
		\includegraphics{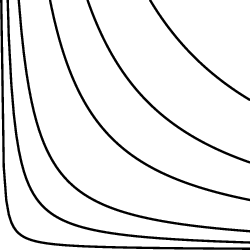}
		\caption{Good}
		\label{fig:good}
	\end{subfigure}
	\begin{subfigure}[b]{0.45\textwidth}
		\centering
		\includegraphics{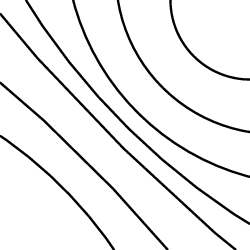}
		\caption{Bad}
		\label{fig:bad}
	\end{subfigure}
	\caption{Null-level sets of $F$ in $\Sn $ for different $x\in \Omega$.}\label{fig:goodbad}
\end{figure}

The proof of Theorem \ref{thm} is concluded in Section \ref{sec:proof} by a rather standard application of the Theorem of Sums. In that sense, \eqref{eq:cond} is not based on any 'new' information as compared to \eqref{eq:classical_structure_cond}. It is just a more efficient interpretation of what is 'already there'. A crucial preliminary step in the proof is to show continuity of the functions $x\mapsto\dist(X,\Theta_\pm(x))$. It turns out that this depends on that the missing symmetry of \eqref{eq:cond} in $\Theta_+$ and $\Theta_-$ is only apparent. Some effort is needed in order to show that \eqref{eq:cond} is equivalent to the analogous statement for $\Theta_-^\delta$ and a large portion of Section \ref{sec:oF} and \ref{sec:Tdelta} is thus devoted to this task. But once established, and since the distance is naturally 1-Lipschitz in $X$, we have enough to make use of the important second order semi jet closures of the sub- and supersolutions.

The conditions in Theorem \ref{thm} may seem a bit technical but they are not harder to verify than e.g. \eqref{eq:classical_structure_cond} or \eqref{eq:Hausdorff_cond}. It is usually very easy to check whether a given equation $F(\cH w,x)=0$ is elliptic at 0, and whether $\partial\Theta_+= \partial\Theta_-\neq \emptyset$. As for the continuity of $\Theta_+^\delta$,
the \emph{standard approach} below seem to work well in many situations.
\begin{enumerate}[1)]
	\item Fix $x_0\in\Omega$ and $t_0>0$ so that $\oB_{t_0}(x_0)\subseteq\Omega$. For $0<t\leq t_0$, let $x,y\in B_t(x_0)$ and define $\delta = \delta(x,y,t) := \frac{|x-y|^2}{t}$. Consider $X\in\Theta_+(x)\cap B_{1/\delta}$ (i.e., $F(X,x)\geq 0$ and $\delta|X|<1$) and write $Z := X(I-\delta X)^{-1}$.
	\item Solve the equation
	\[F(Z + \tau I,y) = 0\]
	for $\tau>0$. Or, if that is not not possible, find an upper bound for a solution of $F(Z + \tau I,y) \geq 0$.
	\item Show that $\tau\to 0$ as $t\to 0$.
\end{enumerate}
This will prove \eqref{eq:cond} since
\[\dist(Z,\Theta_+(y)) = \inf_{Y\in\Theta_+(y)}|Y-Z| \leq |Z+\tau I-Z| = \tau\]
and thus,
\[\lim_{t\to 0^+}\sup_{\substack{x,y\in B_t(x_0)\\ Z\in\Theta^\delta_+(x)}} \dist\left( Z, \Theta_+(y) \right) = 0.\]
In 1) we may assume $\delta>0$ and $F(Z,y) < 0$ since $\dist(Z,\Theta_+(y))$ is zero otherwise. 2): If $F$ is continuous, in addition to elliptic at 0 with $\partial\Theta_+= \partial\Theta_-\neq \emptyset$, then the solution of $F(Z+\tau I,y) = 0$ is exactly $\tau = \dist(Z,\Theta_{+}(y))$. Without continuity we still have $Z + \dist(Z,\Theta_{+}(y))I \in\partial\Theta_\pm(y)$. See Section \ref{sec:oF}. In order to obtain 3), we use the information provided in 1) and the properties of $T_\delta(X)$. Lemma \ref{lem:sm} turns out to be a particularly useful result.

In Section \ref{sec:examples} we give examples on how Theorem \ref{thm} can be applied to prove comparison results in various special cases.
We confirm that \eqref{eq:cond} really is weaker than \eqref{eq:classical_structure_cond} and \eqref{eq:Hausdorff_cond}, and show how the proofs of the classical results are simplified by our approach.
In the classical viscosity theory, the additional non-degeneracy together with \eqref{eq:classical_structure_cond} would typically be a strict monotonicity in the zeroth order term or a strict ellipticity \eqref{eq:strict_elliptic}. Another standard technique is to obtain the comparison principle from \eqref{eq:classical_structure_cond} by perturbing the, say, subsolution $v$ into a \emph{strict} subsolution. In equations without gradient-dependence, an obvious candidate is $v(x)+\frac{\tau}{2}|x|^2$. See Section V in \cite{MR1031377} and Section 5.C in \cite{MR1118699}. For this to be possible, a condition like \eqref{eq:ntg}, i.e. $F(X+\tau I,x)-F(X,x)\geq \ell(\tau)>0$, $\tau>0$, seems necessary. 
We believe that Propositions \ref{prop:BM}-\ref{prop:hausdorff} show that Theorem \ref{thm} covers all the comparable cases in \cite{MR1031377}, \cite{MR1118699}, \cite{MR2084272}, \cite{MR2142457}, \cite{MR2246004}, and \cite{MR3677871}.

We continue by proving the comparison principle for linear equations in Proposition \ref{prop:lin} and for equations that are linear in the eigenvalues of $\cH w$ in Proposition \ref{prop:lineig}. Technically -- via Lemma \ref{lem:sm} and after some manipulation of the equations -- these results follow from Proposition \ref{prop:BM}, but we feel that they are still worth to include because they illustrate very nicely how efficient the standard approach can be. It is also instructive to see that the conditions required for the coefficients present them selves very naturally in the proofs. Moreover, we have not found comparison results for equations of the type \eqref{eq:eigenval_eq} anywhere else in the literature.

Finally we look at equations like $F(\cH w) = f(x)$.
As mentioned above, there is now not much help in \eqref{eq:classical_structure_cond} if not \eqref{eq:ntg} or \eqref{eq:strict_elliptic} holds, and nor in \eqref{eq:Hausdorff_cond} if the null-level sets of $F-f$ are diverging as $|X|\to \infty$. In contrast to
the classical results in Proposition \ref{prop:BM}, \ref{prop:strictly}, \ref{prop:lin}, and  \ref{thm:2} -- and the results in Proposition \ref{prop:hausdorff}, \ref{prop:lineig}, and \ref{prop:semiautonomous}, which probably can be proved using existing ideas as well -- we construct in Proposition \ref{prop:alog_equation} an equation where none of these conditions are met in order to show that \eqref{eq:cond} is capable to also produce truly new results.

A natural question is whether our approach can be generalized to equations that also depend on $\nabla w$ and $w$. A dependence of $w$ can in some sense be a simplifying aspect, and in \cite{MR4147574} the Hausdorff condition \eqref{eq:Hausdorff_cond} is extended to include equations on the form $F(\cH w,w,x) = 0$. It should therefore be possible to extend Theorem \ref{thm} in the same direction, but it is not yet clear what the optimal analogous statement of \eqref{eq:cond} should be.

Recently, it seems that progress have been made in the \emph{potential theoretic} approach on equations like $F(\cH w,\nabla w,w) = f(x)$ in \cite{Harvey2020}. Equations become much harder to handle when the gradient is involved, and they are notoriously difficult to treat under a single theory. It is an interesting question
whether \eqref{eq:cond} has something going for it, also in this direction.
A touchstone would be the $p$-Laplace equation $\di(|\nabla w|^{p-2}\nabla w) = 0$.


\subsection{Definitions and notation}

A function $F$ on $\Sn \times\Omega$
is said to be \emph{elliptic} when
\begin{equation}
X\leq Y\qquad\text{implies}\qquad F(X,x)\leq F(Y,x)\qquad\text{for all $x\in\Omega$.}
\label{eq:elliptic}
\end{equation}
Ellipticity makes viscosity solutions consistent with smooth solutions because of the geometric properties it inflicts on the level sets of $F$. In that respect, \eqref{eq:elliptic} is unnecessarily restrictive since we are only interested in the region where $F$ changes sign. A more appropriate notion is therefore \emph{ellipticity at 0}.
\begin{definition}\label{def:0}
$F\colon \Sn \times\Omega\to \oR$ is \emph{elliptic at 0} if, for all $x\in\Omega$,
\[X\leq Y\quad\text{and}\quad 0 \leq (\text{resp.}<)\;F(X,x)\qquad\text{implies}\qquad 0 \leq (\text{resp.}<)\; F(Y,x).\]
\end{definition}
This is equivalent to requiring the super- and sublevel sets to be \emph{positive} and \emph{negative} \emph{elliptic maps}, respectively:
\begin{equation}\label{eq:elliptic_set_def}
\Theta_\pm(x) + \Sn_\pm = \Theta_\pm(x)\qquad \forall x\in\Omega.
\end{equation}
Here, $\Sn_+$ and $\Sn_-$ are the sets of positive and negative semidefinite matrices.
Clearly, every elliptic operator $F(X,x)$ is elliptic at 0.

\cite{MR2487853} introduces the notion of \emph{duality}. The dual of a subset $\Theta\subseteq \Sn $ is defined as
\begin{equation}
\widetilde{\Theta} := - (\Sn \setminus\Theta^\circ) = \left\{X\in \Sn \;|\; -X\notin \Theta^\circ\right\}.
\label{eq:dual}
\end{equation}
It preserves properness, closedness, and positive ellipticity.
By ellipticity at 0, and with our assumption
\[\Gamma(x) := \partial\Theta_-(x) = \partial\Theta_+(x)\neq\emptyset,\]
we shall show that the closures $\oT_+(x)$ and $-\oT_-(x)$
are proper and positive elliptic duals.
We do not deal with existence in this paper, and it does not affect the validity of the comparison principle that
\[\Gamma_0(x) := \left\{X\in \Sn \;|\; F(X,x) = 0\right\}\]
can be empty. The more relevant object is $\Gamma(x)$, which we, somewhat inaccurately, call 
\emph{the null-level set} of $F$ at $x$.

Besides being discontinuous, we allow the operator to take values in the extended real line $\oR := \mR\cup\{-\infty,\infty\}$.
This is really just a practical consideration in that we do not have to care about constructions like \emph{elliptic branches} or \emph{admissible solutions}. Instead we can extend $F$ to be infinite in the non-elliptic or 'uninteresting' regions of $\Sn \times\Omega$.
Thus, equations with up to two constrains
\begin{equation}\label{eq:system_with_constrains}
F(\cH w,x) = 0,\qquad \cH w\notin \Psi(x),\,\cH w\notin \Phi(x),
\end{equation}
can be integrated into our general framework provided $\Psi$ and $\Phi$ are negative- and positive elliptic maps, respectively, with $\Psi(x)\cap\Phi(x)=\emptyset$. The system \eqref{eq:system_with_constrains} is then reformulated as $F'(\cH w,x)=0$, where $F'$ is defined on $\Sn\times\Omega$ as
\[F'(X,x) :=
\begin{cases}
-\infty,\qquad &\text{if $X\in \Psi(x)$,}\\
\infty,\qquad &\text{if $X\in \Phi(x)$,}\\
F(X,x), &\text{otherwise,}
\end{cases}\]


We end the Introduction with some words about our notation. Unless otherwise stated, $\Omega$ denotes an open and bounded subset of Euclidean space $\mR^n$.
The second order partial derivatives of a function $w\in C^2(\Omega)$ are gathered in the Hessian matrix $\cH w\colon\Omega\to \Sn $, and the equations are expressed as $F(\cH w,x) = 0$ instead of the more cumbersome $F(\cH w(x),x) = 0$.
The inequality $X<Y$ ($X\leq Y$) is the standard partial ordering in the space $\Sn $ of symmetric $n\times n$ matrices and means that $Y-X$ is positive (semi)definite. Moreover, $\Sn $ is equipped with the inner product $\tr(XY)$ and the induced norm is thus
\[\lVert X\rVert := \lVert X\rVert_2 := \sqrt{\tr(X^2)} = \sqrt{\lambda_1^2(X) + \cdots + \lambda_n^2(X)}\]
where $\lambda_1(X)\leq\cdots\leq\lambda_n(X)$ are the eigenvalues of $X$.
However, in \eqref{eq:Tddef}, \eqref{eq:distdef}, and thus implicitly in \eqref{eq:cond}, we use the \emph{operator norm}
\[|X| := \lVert X\rVert_\infty := \max\{-\lambda_1(X),\lambda_n(X)\}.\]
It is the better match to ellipticity.
We have
\[|X| \leq \lVert X\rVert \leq \lVert X\rVert_1\leq \sqrt{n}\lVert X\rVert \leq n|X|,\]
where $\lVert X\rVert_1 := \sum_{i=1}^n|\lambda_i(X)|$ is the \emph{trace norm}.
Next, $USC(\oO)$ and $LSC(\oO)$ are the spaces of upper- and lower semicontinuous functions on the closure of $\Omega$, respectively. For an account on distance functionals on metric spaces, like the excess and the Hausdorff distance, we refer to the book \cite{MR1269778}. The definition of viscosity solutions can be found in \cite{MR1118699}, which will also be our main reference to the subject.

\section{Examples of application}\label{sec:examples}

Suppose that an operator $F$ satisfies the classical structure condition \eqref{eq:classical_structure_cond} and that
\begin{equation}
F(X+\tau I,x) - F(X,x) \geq \ell(\tau),\qquad x\in\Omega,\,\tau\geq 0,\,X\in\Sn,
\label{eq:ntg}
\end{equation}
for some strictly increasing $\ell\geq 0$.
If $v$ is a subsolution to the equation $F(\cH w,x) = 0$, then $v(x) + \frac{\tau}{2}|x|^2$ is a strict subsolution for every $\tau>0$ and comparison is obtained from the proof of Theorem 3.3 in \cite{MR1118699}, taking into account the comments in Section 5.C there. A version of this result is explicitly stated in
\cite{MR2246004} where an operator is called \emph{non-totally degenerate} if there is a constant $\ell>0$ such that
$F(X+\tau I,\cdots) - F(X,\cdots) \geq \ell \tau$.
By ellipticity, note that \eqref{eq:ntg} is weaker than a \emph{directional non-degeneracy} $F(X+\tau \xi\xi^T,x) - F(X,x) \geq \ell(\tau)$ that is used by some authors.

As shown below, our condition \eqref{eq:cond} follows very easily from \eqref{eq:classical_structure_cond} and \eqref{eq:ntg}, thereby providing a simple proof of the comparison principle. Observe that \eqref{eq:classical_structure_cond} is only required to hold in compact subsets of $\Omega$. This improvement is well known, although an exact reference is hard to find.

%

\begin{proposition}[Non-totally degenerate case]\label{prop:BM}
Let $F\colon \Sn \times\Omega\to\mR$ be elliptic at 0 and satisfy \eqref{eq:ntg}. If there is a modulus of continuity $\omega$ such that
\begin{equation}\label{eq:ntg_cont_cond}
F(X,x) - F(Y,y) \leq \omega\Big(\alpha|x-y|^2 + |x-y| \Big)
\end{equation}
whenever $\alpha>0$, $x,y$ is in a compact subset of $\Omega$, and
\[
-\alpha\begin{bmatrix}
I & 0\\ 0 & I
\end{bmatrix} \leq	\begin{bmatrix}
	X & 0\\ 0 & -Y
	\end{bmatrix} \leq \alpha  \begin{bmatrix}
	I & -I\\ -I & I
	\end{bmatrix},
	\label{eq:str}
\]
then the comparison principle holds for the equation $F(\cH w,x) = 0$ in $\Omega$.
\end{proposition}

\begin{proof}
$\emptyset\neq\partial\Theta_+= \partial\Theta_-$ in $\Omega$ because of \eqref{eq:ntg}. We take the standard approach as described in the Introduction.

Since $\ell$ is strictly increasing, it has a non-decreasing inverse $\tau(s) := \ell^{-1}(s)$ with $\tau(0^+) = 0$. 
Thus, with $s := -F(Z,y) > 0$ we get
\[F\left( Z + \tau(s)I,y\right) \geq \ell(\tau(s)) + F(Z,y) = 0,\]
and $Z + \tau(s)I\in\Theta_+(y)$.
By Corollary \ref{cor:matrixineq}, we find that $X,Z$ satisfies \eqref{eq:str} with $\alpha = 1/\delta = t/|x-y|^2 > |X|$.
As $F(X,x)\geq 0$, $x,y\in B_{t}(x_0)\subset\subset\Omega$, it follows that
\begin{align*}
\dist\left( Z, \Theta_+(y) \right)
	&\leq \tau\Big(-F\left(Z,y\right)\Big)\\
	&\leq \tau\Big(F(X,x)-F\left(Z,y\right)\Big)\\
	&\leq \tau\Big( \omega\Big(\frac{1}{\delta}|x-y|^2 + |x-y| \Big) \Big)\\
	&\leq (\tau\circ \omega) (3t),
	\end{align*}
	which goes to zero when $t\to 0$.
\end{proof}



The strictly elliptic case is considered in e.g. \cite{MR2084272}. By applying Theorem III.1(1) in \cite{MR1031377} on compact subsets of $\Omega$, and perturbing to strict sub- or supersolutions, it is clear that Proposition 3.8 in \cite{MR2084272} can be improved: The ellipticity constant $\lambda>0$ may be allowed to depend on $x$ as long as it is \emph{locally} bounded away from 0.
For example, one may require that there exists a positive $\lambda\in LSC(\Omega)$ such that
\begin{equation}\label{eq:strict_elliptic}
F(Y,x) - F(X,x) \geq \lambda(x)\tr(Y-X)
\end{equation}
for all $x\in\Omega$ and all $X,Y\in \Sn $ with $X\leq Y$.

Apart from the $x$-dependence one should note that \eqref{eq:strict_elliptic} is considerably more restrictive than \eqref{eq:ntg}. On the other hand, the continuity condition \eqref{eq:strict_cont_cond} now required for $F$ is weaker than \eqref{eq:ntg_cont_cond} in Proposition \ref{prop:BM}.

\begin{proposition}[Strictly elliptic case]\label{prop:strictly}
Let $F\colon \Sn \times\Omega\to\mR$ satisfy \eqref{eq:strict_elliptic}. If there is a modulus of continuity $\omega$, so that
\begin{equation}\label{eq:strict_cont_cond}
\left|F(X,x)-F(X,y)\right| \leq \omega\Big( (\lVert X\rVert + 1)|x-y| \Big)
\end{equation}
whenever $X\in \Sn$ and $x,y$ is in a compact subset of $\Omega$,
then the comparison principle holds for the equation $F(\cH w,x) = 0$ in $\Omega\subseteq\mR^n$.
\end{proposition}

In order to obtain step 3) in the standard approach, the inequality
\begin{equation}
X(I-\delta X)^{-1} \geq X + \frac{\delta}{2}X^2,\qquad  \delta\geq 0,\,\delta |X|<1, 
\label{eq:Xdineq}
\end{equation}
will be convenient to have at hand. See \eqref{eq:T_firstorder_approx}.

\begin{proof}
$F$ is elliptic at 0 and $\emptyset\neq\partial\Theta_+(x)= \partial\Theta_-(x)$ by the strict ellipticity.

The equation $F(Z+\tau I,y) = 0$ yields
\[0 = F(Z+\tau I,y)\\
\geq F(X,y) + \lambda(y)\tr(Z+\tau I- X),\]
which implies
\begin{align*}
n\lambda(y)\tau
	&\leq - F(X,y) - \tr(Z-X)\\
	&\leq F(X,x) -F(X,y) - \lambda(y)\tr\left(\frac{\delta}{2} X^2\right), &&\text{(by \eqref{eq:Xdineq})}\\
	&\leq \omega\Big( (\lVert X\rVert + 1)|x-y| \Big) - \lambda(y)\frac{|x-y|^2}{2t}\lVert X\rVert^2.
\end{align*}
The modulus $\omega$ may be assumed to be subadditive, and for $\epsilon>0$ there is always a constant $M_\epsilon>0$ such that $\omega(r)\leq \epsilon + M_\epsilon r$. When writing $r := \lVert X\rVert |x-y|$, completing the square, and using that $|x-y|<2t$, the above is bounded by
\begin{align*}
\omega(|x-y|) + \omega(r) - \frac{\lambda(y)}{2t}r^2
	&\leq \omega(2t) + \epsilon + M_\epsilon r - \frac{\lambda(y)}{2t}r^2\\
	&\leq \omega(2t) + \epsilon + \frac{M_\epsilon^2}{2\lambda(y)}t.
\end{align*}
Thus \eqref{eq:cond} holds since $\lim\sup_{y\to x_0}\lambda^{-1}(y)\leq \lambda^{-1}(x_0)$, and
\[
\lim_{t\to 0^+}\sup_{\substack{x,y\in B_t(x_0)\\ Z\in\Theta^\delta_+(x)}} \dist\left( Z, \Theta_+(y) \right)\\
	\leq \lim_{t\to 0^+}\sup_{y\in B_t(x_0)} \frac{\omega(2t) + \epsilon}{n\lambda(y)} + \frac{M_\epsilon^2}{2n\lambda^2(y)}t \leq \frac{\epsilon}{n\lambda(x_0)}
\]
for every $\epsilon>0$.
\end{proof}

With a given parameter $R>0$, the \emph{bounded Hausdorff distance} between subsets $\Theta_1$ and $\Theta_2$ of $\Sn $ is
\[\dd_R(\Theta_1,\Theta_2) := \sup_{|X| < R}\left|\dist(X,\Theta_1) - \dist(X,\Theta_2)\right|.\]
Since, obviously, $\dd_R\leq \dd_\infty$, the result of \cite{MR3677871} is a direct consequence of our next application of Theorem \ref{thm}. It presents a condition intermediate to \eqref{eq:Hausdorff_cond} and \eqref{eq:cond}.

\begin{proposition}\label{prop:hausdorff}
Let $F\colon \Sn \times\Omega \to \oR$ be elliptic at 0. If for each $x_0\in\Omega$, $\emptyset \neq\partial\Theta_+(x_0)= \partial\Theta_-(x_0)$ and
\begin{equation}
\lim_{t\to 0^+}\sup_{x,y\in B_t(x_0)} \dd_{t|x-y|^{-2}}\Big(\Theta_+(x),\Theta_+(y)\Big) = 0,
\label{eq:int}
\end{equation}
then the comparison principle holds for the equation $F(\cH w,x) = 0$ in $\Omega$.
\end{proposition}

\begin{proof}
The bounded Hausdorff distance can be alternatively expressed as
\[\dd_R(\Theta_+(x),\Theta_+(y)) = \max\left\{\sup_{\substack{X\in\Theta_+(x)\\\ |X| < R}}\dist(X,\Theta_+(y)),\; \sup_{\substack{Z\in\Theta_+(y)\\ |Z| < R}}\dist(Z,\Theta_+(x))\right\},\]
and in the next Section we will show that the negative of the distance to $\Theta_+(y)$ is an elliptic function. That is, $\dist\left( X(I-\delta X)^{-1}, \Theta_+(y) \right) \leq \dist\left( X, \Theta_+(y) \right)$ since $X(I-\delta X)^{-1}\geq X$.
Thus,
\begin{align*}
\lim_{t\to 0^+}\sup_{\substack{x,y\in B_t(x_0)\\ Z\in\Theta^\delta_+(x)}} \dist\left( Z, \Theta_+(y) \right)
	&\leq \lim_{t\to 0^+}\sup_{x,y\in B_{t}(x_0)}\sup_{\substack{X\in \Theta_+(x)\\ |X| < \frac{t}{|x-y|^2}}} \dist\left( X, \Theta_+(y) \right)\\
	&\leq \lim_{t\to 0^+}\sup_{x,y\in B_{t}(x_0)}\dd_{t|x-y|^{-2}}\left(\Theta_+(x),\Theta_+(y)\right) = 0.
\end{align*}
\end{proof}

\begin{example}\label{ex:determinat_eq_not_classical}
Let $M\colon\Omega\to\Sn$ and $f\colon\Omega\to[0,\infty)$. \cite{MR3677871} consider the perturbed Monge-Ampère operator
\[F(X,x) := \begin{cases}
\det(X+M(x)) - f(x),\qquad & \text{if $X+M(x)\geq 0$},\\
-\infty, &\text{otherwise.}
\end{cases}\]
It satisfies the Hausdorff continuity condition \eqref{eq:Hausdorff_cond} for uniformly continuous data, and it satisfies \eqref{eq:int} when $M$ and $f$ are continuous. The comparison principle thus holds for the equation $F(\cH w,x) = 0$.
However, the classical structure condition does not hold and the equation is therefore not covered by the techniques presented in \cite{MR1118699}. In particular, it is not covered by Proposition \ref{prop:BM}. For convenience, we repeat a special case of their counterexample (Remark 5.10 \cite{MR3677871}) here.

Let $\Omega\subseteq\mR^2$ be an open and bounded set containing the origin. Define $X_k := \bigl[\begin{smallmatrix}
0 & 0\\ 0 & k/2
\end{smallmatrix}\bigr]$, $Z_k := T_{1/k}(X_k) = \bigl[\begin{smallmatrix}
0 & 0\\ 0 & k
\end{smallmatrix}\bigr]$, and $x_k = \frac{2}{k}e$ where $e$ is a unit vector in $\mR^2$. Now, $k|x_k-0|^2 + |x_k-0| \overset{k\to\infty}{\to} 0$ but if we let $M(x) := \bigl[\begin{smallmatrix}
|x| & 0\\ 0 & 0
\end{smallmatrix}\bigr]$ then
\begin{align*}
F(X_k,x_k) - F(Z_k,0)
	&= \det(X_k+M(x_k)) - \det(Z_k+M(0)) + f(0) - f(x_k)\\
	&\geq 1 - 0 - \omega_f(2/k) \to 1 > 0.
\end{align*}

\end{example}

The condition \eqref{eq:int} is more general than \eqref{eq:Hausdorff_cond}, but it is still not capable to handle linear equations in a satisfactory way. Indeed, given an equation $\tr(A(x)\cH w) = 0$, where we assume for simplicity that $\tr A(x)\equiv 1$, one can show that the bounded Hausdorff distance between two superlevel sets is $\dd_R(\Theta_+(x),\Theta_+(y)) = R| A(x)-A(y)|$. Thus, for $t>0$,
\[\sup_{x,y\in B_{t}(x_0)} \dd_{t|x-y|^{-2}}\Big(\Theta_+(x),\Theta_+(y)\Big) = t\sup_{x,y\in B_{t}(x_0)}\frac{| A(x)-A(y)|}{|x-y|^2},\]
which cannot be assumed to be finite for all $x_0$ unless $A$ is constant.

In order to prove the linear case with our approach we need a little lemma.

\begin{lemma}\label{lem:sm}
Let $X\in \Sn $ and $\delta>0$. If
$\delta X < I$, then
\[Q_1^TXQ_1 - Q_2^TX\left(I-\delta X\right)^{-1}Q_2 \leq \frac{1}{\delta} (Q_1-Q_2)^T(Q_1-Q_2)\]
in $\mathcal{S}^m$ for all $n\times m$ matrices $Q_1,Q_2$. ($m\in\mathbb{N}$).
\end{lemma}
Its proof consists of multiplying the inequality in Corollary \ref{cor:matrixineq} from the left and right by $\bigl[\begin{smallmatrix}
Q_1^T & Q_2^T
\end{smallmatrix}\bigr]$ and $\bigl[\begin{smallmatrix}
Q_1 \\ Q_2
\end{smallmatrix}\bigr]$, respectively.

\begin{proposition}[Linear elliptic case]\label{prop:lin}
The comparison principle holds for the equation
\[\tr(A(x)\cH w) = f(x) \qquad\text{in $\Omega$}\]
whenever $A\in C(\Omega;\Sn_+ \setminus\{0\})$, $f\in C(\Omega)$,
and whenever there is a locally Lipschitz $Q\colon\Omega\to \mR^{n\times m}$ such that
\begin{equation}
\frac{A(x)}{\|A(x)\|_1} = Q(x)Q^T(x).
\label{eq:lincon}
\end{equation}
\end{proposition}

Again the result can be deduced from the ideas of Section 5.C in \cite{MR1118699}, but it does not seem to have been explicitly stated before.

As shown in Section \ref{sec:lin} the requirement $A(x)\neq 0$ is necessary when the equation is without first- and zeroth order terms. In addition, the comparison principle may fail also if $Q$ is only Hölder continuous.

It is known that the existence of a Lipschitz decomposition $A=QQ^T$ is enough in order for the equation to satisfy the structure condition \eqref{eq:classical_structure_cond}. 
When using the standard approach, the variant \eqref{eq:lincon} appears in the proof. It is a tad weaker because any non-Lipschitz behaviour in the normal direction of $A$ is cancelled out. That is, the Proposition may apply to the equation $\tr(A(x)\cH w) = f(x)$ even if $x\mapsto \tr A(x) = \|A(x)\|_1$ is not Lipschitz -- which is rather natural in view of the equivalent equation $\tr((A(x)/\|A(x)\|_1)\cH w) = f(x)/\|A(x)\|_1$. Observe that only in this form the equation satisfy the non-degeneracy \eqref{eq:ntg} of Proposition \ref{prop:BM}.

\begin{proof}
As $A(x)\geq 0$ the equation is elliptic, and $\emptyset\neq\partial\Theta_+(x)= \partial\Theta_-(x)$ because  $A(x)\neq 0$.

The distance from a point to the level set hyperplane of a linear function is provided by the \emph{Ascoli formula},
\[\tau := \dist(Z,\Theta_+(y)) = \dist(Z,\Gamma(y)) = \frac{\left|\tr(A(y)Z) - f(y)\right|}{\|A(y)\|_1}.\]
Writing $\hat{A}(x) := \frac{A(x)}{\|A(x)\|_1}$ and $\hat{f}(x) := \frac{f(x)}{\|A(x)\|_1}$, and using that
$\tr (A(y)Z) - f(y)< 0\leq \tr (A(x)X) - f(x)$ in the standard approach, the Lemma yields
\begin{align*}
\tau &= -\tr\left(\hat{A}(y)Z\right) + \hat{f}(y)\\
	 &\leq \tr\left(\hat{A}(x)X\right) - \hat{f}(x) -\tr\left(\hat{A}(y)Z\right) + \hat{f}(y)\\
	 &= \tr\Big(Q^T(x)XQ(x)  - Q^T(y)ZQ(y)\Big) + \hat{f}(y) - \hat{f}(x)\\
	 &\leq \frac{1}{\delta}\left\lVert Q(x) - Q(y)\right\rVert^2 + \omega(|x-y|)
\end{align*}
where $\omega$ is the modulus of continuity of $\hat{f}$ in $\oB_{t_0}(x_0)$. Since $\delta = |x-y|^2/t$, the first term on the right-hand side is bounded by $L^2 t$ where $L$ is the Lipschitz constant of $Q$ in $\oB_{t_0}(x_0)$.
Thus,
\begin{equation}\label{eq:linear_conclusion}
\lim_{t\to 0^+}\sup_{\substack{x,y\in B_t(x_0)\\ Z\in\Theta^\delta_+(x)}} \dist\left( Z, \Theta_+(y) \right) \leq \lim_{t\to 0^+} L^2 t + \omega(2t) = 0.
\end{equation}
\end{proof}

The proof for operators on the form $\bfa^T(x)\overline{\lambda}(X) - f(x)$ is almost identical: Write $\bfa(x) = [a_1(x),\dots,a_n(x)]^T$, $|\bfa(x)|_1 := \sum_{i=1}^n |a_i(x)|$, and $\mR^n_+ := \{z\in\mR^n\;|\;z_i\geq 0\}$.

\begin{proposition}\label{prop:lineig}
	The comparison principle holds for the equation
\begin{equation}\label{eq:eigenval_eq}
\sum_{i=1}^n a_i(x)\lambda_i(\cH w) = f(x) \qquad\text{in $\Omega$}
\end{equation}
	whenever $\bfa\in C(\Omega;\mR^n_+\setminus\{0\})$, $f\in C(\Omega)$,
	and whenever the square roots of
\[\hat{a}_i(x) := \frac{a_i(x)}{|\bfa(x)|_1},\qquad i = 1,\dots,n,\]
	are locally Lipschitz in $\Omega$.
\end{proposition}

Put differently, comparison holds for every continuous $f$ when each $a_i\geq 0$ is locally Lipschitz, not all being zero at once, and if $a_i(x_0) = 0$ then $a_i(x) = \mathcal{O}(|x-x_0|^2)$ as $x\to x_0$.
\begin{proof}
	The equation is elliptic as $a_i\geq 0$ and because $X\mapsto\lambda_i(X)$ is elliptic (Corollary III.2.3, \cite{MR1477662}). Furthermore, $\emptyset\neq\partial\Theta_+(x)= \partial\Theta_-(x)$ because $\bfa(x)\neq 0$.

The standard approach yields
\[0 = F(Z+\tau I,y) = \sum_{i=1}^n a_i(y)(\lambda_i(Z) + \tau) - f(y) = F(Z,y) + |\bfa(y)|_1\tau,\]
so
\begin{align*}
\tau &= -\frac{1}{|\bfa(y)|_1}F(Z,y)\\
	 &\leq \frac{1}{|\bfa(x)|_1}F(X,x) -\frac{1}{|\bfa(y)|_1}F(Z,y)\\
	 &= \sum_{i=1}^n\hat{a}_i(x)\lambda_i(X) - \hat{a}_i(y)\lambda_i(Z) + \hat{f}(y) - \hat{f}(x),\qquad \hat{f}:= f/|\bfa|_1,\\
	 &\leq \frac{1}{\delta}\sum_{i=1}^n \left(\sqrt{\hat{a}_i(x)} - \sqrt{\hat{a}_i(y)}\right)^2 + \omega_{\hat{f}}(|x-y|)
\end{align*}
by the case $n=m=1$ of Lemma \ref{lem:sm}. The conclusion is \eqref{eq:linear_conclusion} with $L^2 := \sum_{i=1}^n L_i^2$ where $L_i$ is the Lipschitz constant of $\sqrt{\hat{a}_i}$.
\end{proof}

%

We now turn to the \emph{semiautonomous} case. That is, when the dependence of the operator in $X$ and $x$ can be separated as $F(X) - f(x)$.
In the simplest situation $f\equiv 0$, we observe that the condition \eqref{eq:cond} is automatically fulfilled since the sub- and superlevel sets
\[\Theta_- := \left\{X\in \Sn \;|\; F(X)\leq 0\right\},\quad \Theta_+ := \left\{X\in \Sn \;|\; F(X)\geq 0\right\},\]
are constant.
Indeed, ellipticity at 0 and the inequality $X(I-\delta X)^{-1}\geq X$ yields $\Theta_+^\delta\subseteq\Theta_+$ so
$\Theta_+^{\delta}$ does not exceed $\Theta_+$ for any $\delta\geq 0$.
We are only left with the requirement $\partial\Theta_- = \partial\Theta_+ \neq\emptyset$. But that is, as we shall see, a necessary condition as well. It is safe to say that if you are not able to provide an immediate counterexample, then the comparison principle \emph{always} holds for autonomous equations $F(\cH w) = 0$ elliptic at 0. It seems that this is the only case where the problem of comparison is \emph{solved}. That is, the sufficient conditions are proved to also be necessary. This was essentially achieved in \cite{MR2487853} although the necessity of their condition was not pointed out.

\begin{proposition}[Autonomous case]\label{thm:2}
Let $0\not\equiv F\colon \Sn \to\oR$ be elliptic at 0. Then the following are equivalent.
\begin{enumerate}[(i)]
	\item The comparison principle holds for the equation \[F(\cH w) = 0\] in any open and bounded $\Omega\subseteq\mR^n$.
	\item The set $\Gamma_0 := \left\{X\in \Sn \;|\; F(X) = 0\right\}$ does not contain an open ball.
	\item For all $X\in \Sn $,
	\[\sup\{t\;|\; X+tI\in\Theta_-\} = \inf\{t\;|\; X+tI\in\Theta_+\}.\]
	\item $F(X_0-tI)<0<F(X_0+tI)$ for all $X_0\in\Gamma_0$ and all $t>0$.
	\item The sets $\oT_+$ and $-\oT_-$ are duals. (See \eqref{eq:dual})
	\item $\partial\Theta_- = \partial\Theta_+$.	
\end{enumerate}
\end{proposition}

\begin{proof}
The necessity of part (ii) in order to get the comparison principle is quite apparent because we otherwise could construct two different polynomial solutions $\phi(x) = \frac{1}{2}x^T X_0 x$ and $\psi(x) = \phi(x) + \epsilon(1 - \lvert x\lvert^2)$ with equal boundary values at $\lvert x\rvert = 1$. Thus, (i) $\Rightarrow$ (ii).

Assume that $M:=\sup\{t\;|\; X+tI\in\Theta_-\}$ is not equal to $m:= \inf\{t\;|\; X+tI\in\Theta_+\}$ for some $X$. One can check that it is impossible to have $M<m$. We may therefore pick two real numbers $a$ and $b$ such that $m\leq a<b\leq M$.
But then $F(X+aI)\geq 0$ and $F(X+bI)\leq 0$ and the whole open set $\{Y\;|\;X + aI < Y < X+bI\}$ will be in $\Gamma_0$ since $F$ is elliptic at 0. That is, (ii) $\Rightarrow$ (iii).

Next, for $X_0\in\Gamma_0 = \Theta_-\cap\Theta_+$  we have $\sup\{t\;|\; X_0+tI\in\Theta_-\}\geq 0$ and $\inf\{t\;|\; X_0+tI\in\Theta_+\}\leq 0$. Thus, when $t>0$, (iii) implies $X_0 - tI\notin\Theta_+$ and $X_0 + tI\notin\Theta_-$. That is, $F(X_0-tI)<0<F(X_0+tI)$ and we have (iii) $\Rightarrow$ (iv).

Suppose now that $\oT_+$ and $-\oT_-$ are not duals. By Lemma \ref{lem:Thetaprop} (3),
\[\oT_+ \neq \Sn \setminus(\oT_-)^\circ = \Sn \setminus\Theta_-^\circ = \partial \Theta_-\cup (\Theta_+\setminus\Gamma_0).\]
From part (1) of the same Lemma we always have $\partial\Theta_-\subseteq\oT_+$, so it must be the right inclusion that does not hold. That is, there is a matrix $Z$ in $\oT_+$ but not in $\Sn \setminus\Theta_-^\circ$. In other words, $Z\in\oT_+\cap\Theta_-^\circ$. Since $\Theta_+^\circ$ is obviously open, part (2) of the Lemma yields $Z+\epsilon I\in\Theta^\circ_+\cap\Theta_-^\circ\subseteq\Gamma_0$ for all sufficiently small $\epsilon>0$. This contradicts (iv) and we have proved (iv) $\Rightarrow$ (v).

By (v) we get
\[\partial\Theta_+ = \oT_+\setminus\Theta_+^\circ = (\Sn \setminus\Theta_-^\circ)\setminus\Theta_+^\circ = \Sn \setminus(\Theta_-^\circ\cup\Theta_+^\circ),\]
which is symmetric in $+$ and $-$. Thus, (v) $\Rightarrow$ (vi).

Finally, (i) follows from (vi) by Theorem \ref{thm} when $\partial\Theta_\pm$ is nonempty. However, if $\partial\Theta_- = \partial\Theta_+ = \emptyset$, a non-trivial operator is either negative or positive in $\Sn $ and the comparison principle holds vacuously since the equation $F(\cH w) = 0$ will only have one type of solutions. The proof of the Proposition is therefore completed as (vi) $\Rightarrow$ (i).
\end{proof}

In the next result we present a simple condition on $F$ that will ensure the comparison principle also for equations $F(\cH w) = f(x)$ with a non-constant right-hand side. The condition is rather crude, but, on the other hand, easy to verify. We mention that an operator $F\colon \mathbf{F}\to\mR$, $\mathbf{F}\subseteq \Sn \times\mR^n\times\mR$, is called \emph{tame} in \cite{MR4286829} if there is a positive function $\ell$ such that $F(X+\tau I,p,s-r) - F(X,p,s) \geq \ell(r,\tau)$ for all $(X,p,s)\in \mathbf{F}$ and all $r,\tau>0$. This is obviously a condition very similar to the non-totally degenerate condition in \cite{MR2246004}. In \cite{MR4286829}, $F$ does not have to be defined on all of $\Sn \times\mR^n\times\mR$ and the ellipticity is given in terms of a \emph{monotonicity cone}. Our result below is a generalization of the uniqueness part of Theorem 2.7 \cite{MR4286829} in the case where $\mathbf{F}\subseteq \Sn $.


\begin{proposition}[Semiautonomous case]\label{prop:semiautonomous}
Let $f\in C(\Omega)$ and $F\colon \Sn \to\oR$ such that $F - f$ is elliptic at 0 with $\partial\Theta_+(x) = \partial\Theta_-(x)\neq\emptyset$.
The comparison principle holds for the equation
\[F(\cH w) = f(x)\qquad\text{in $\Omega$,}\]
if whenever $X_k\in\Sn$ is a sequence such that $F(X_k)\to f(x_0)$ then
\begin{equation}\label{eq:semiautcond}
\liminf_{k\to\infty}F\left( X_k + \tau I\right) > f(x_0)
	\end{equation}
for all $\tau>0$.
\end{proposition}

\begin{proof}
Suppose \eqref{eq:cond} is not true. i.e.,
\[\lim_{t\to 0^+}\sup_{\substack{x,y\in B_t(x_0)\\ Z\in\Theta_+^\delta(x)}}\dist(Z,\Theta_+(y)) > 0,\qquad \delta := \tfrac{|x-y|^2}{t},\]
for some $x_0\in\Omega$. Then there is a sequence $t_k\searrow 0$, points $x_k,y_k\in B_{t_k}(x_0)$, $X_k\in \Theta_+(x_k)\cap B_{1/\delta_k}$, and a positive $\tau$ such that
\[\dist(X_k,\Theta_+(y_k))\geq\dist(T_{\delta_k}(X_k),\Theta_+(y_k))\geq\tau\]
for all $k$. Since $F(X_k)\geq f(x_k)$, $F(T_{\delta_k}(X_k))<f(y_k)$, and $X_k\leq T_{\delta_k}(X_k)$, ellipticity at 0 yields $f(x_k)\leq F(X_k)<f(y_k)$ and thus $F(X_k)\to f(x_0)$.

By Corollary \ref{cor:oF} (1. and 3.), $X_k + \dist(X_k,\Theta_+(y_k))I\in\Gamma(y_k)$. Therefore,
\begin{align*}
&&f(y_k) &> F\left(X_k + \dist(X_k,\Theta_+(y_k))I - \frac{\tau}{2}I\right),\\
&\Rightarrow &f(y_k) &> F\left(X_k + \frac{\tau}{2}I\right),
\end{align*}
and taking the $\liminf$ then contradicts \eqref{eq:semiautcond}.
\end{proof}


\begin{example}
The \emph{special Lagrangian potential equation}
\[F(\cH w) := \sum_{i=1}^n \arctan(\lambda_i(\cH w)) = \theta\]
has attained much interest since it was introduced in \cite{MR666108}.
For a right-hand side constant $\theta \in (-n\pi/2,n\pi/2)$ the solutions have a nice geometrical interpretation. The graph of the gradient $\nabla w$ in $\Omega\times \mR^n$ is a \emph{special Lagrangian manifold}. i.e., it is a Lagrangian manifold of minimal area. 
See \cite{MR4179860} and the references therein.

The comparison principle is immediate by Proposition \ref{thm:2} and recently, \cite{MR4147574} were able to extend the result, using \eqref{eq:Hausdorff_cond}, to the equation
\begin{equation}\label{eq:SLPequation}
\sum_{i=1}^n \arctan(\lambda_i(\cH w)) = f(x)
\end{equation}
when $f\colon\Omega\to (-n\pi/2,n\pi/2)$ is continuous and avoids the \emph{special phase values}
\[\theta_j := (n-2j)\frac{\pi}{2},\qquad j = 1,\dots,n-1.\]
Note that Proposition \ref{prop:BM} is useless in this situation because $F$ is bounded and any property like $F(X+\tau I)-F(X)\geq \ell(\tau)$ is out of the question.

We bring up this equation because Proposition \ref{prop:semiautonomous} provides a very simple proof of the comparison result. Indeed, if \eqref{eq:semiautcond} is not true, then there are numbers $\theta\in f(\Omega)$, $\tau>0$, and a sequence $X_k\in \Sn$ with $F(X_k)\to \theta$, such that
\begin{align*}
0 &= \lim_{k\to\infty} F(X_k + \tau I) - F(X_k)\\
  &= \lim_{k\to\infty} \sum_{i=1}^n \arctan(\lambda_i(X_k) + \tau) - \arctan(\lambda_i(X_k))\geq 0,
\end{align*}
which -- since $\arctan$ is strictly increasing -- is possible only if each $\lambda_i(X_k)$ is unbounded as $k\to\infty$. There is thus a subsequence (still indexed by $k$) such that either $\lambda_i(X_k)\to +\infty$ or $\lambda_i(X_k)\to -\infty$. But this is a contradiction of the assumptions as
\[f(\Omega)\ni\theta = \lim_{k\to\infty} F(X_k)\\
= \lim_{k\to\infty} \sum_{i=1}^n  \arctan(\lambda_i(X_k))\\
= \sum_{i=1}^n \pm\frac{\pi}{2}
= \theta_j
\]
for some $j=0,\dots,n$.

The result is sharp in the sense that for each $j=1,\dots,n-1$ there exists a continuous $f$ with $\theta_j\in f(\Omega)\subseteq (-n\pi/2,n\pi/2)$ so that the comparison principle does not hold for the equation \eqref{eq:SLPequation}. See \cite{arxiv.2206.09373}.
\end{example}

We conclude this Section with an equation that to our knowledge cannot be covered by the existing theory.
It has the same form as the special Lagrangian potential equation, but it is, in a sense, somewhat less degenerate. Recall that
\[\arctan(\lambda) = \int_0^\lambda\frac{\dd s}{1 +s^2}.\]
Consider instead
\[a(\lambda) := \int_0^\lambda\frac{\dd s}{1 +|s|} = \sgn(\lambda)\ln(1+|\lambda|)\]
and the equation
\begin{equation}\label{eq:log_equation}
\sum_{i=1}^n a(\lambda_i(\cH w)) = f(x).
\end{equation}
The operator $F(X) := \sum_{i=1}^n a(\lambda_i(X))$ is unbounded, which generally is a good thing. However, Proposition \ref{prop:BM} is still useless and now also \eqref{eq:semiautcond} in Proposition \ref{prop:semiautonomous} will fail for every value $f(x_0)\in\mR$. This is because it is always possible to construct a sequence $X_k$ such that each $\lambda_i(X_k)$ goes to $\infty$ or $-\infty$ while still $F(X_k) \to f(x_0)$. Thus, $F(X_k+\tau I)-F(X_k)\to 0$ since $a'(\lambda)\to 0$ as $\lambda\to\pm\infty$. Of course, $F(X+\tau I)-F(X)>0$ for all $X\in\Sn$ and $\tau>0$, but this is not enough to construct strict sub- or supersolutions. At least not as they are defined in (5.2) in \cite{MR1031377}. 

The Hausdorff continuity condition \eqref{eq:Hausdorff_cond} works for the special Lagrangian equation because the level sets of the operator are asymptotically parallel in each interval $(\theta_j,\theta_{j-1})$. There is no such behaviour in the equation \eqref{eq:log_equation}. See Figure \ref{fig:a}.
\begin{figure}[h]
	\centering
	\begin{subfigure}[b]{0.45\textwidth}
		\centering
		\includegraphics{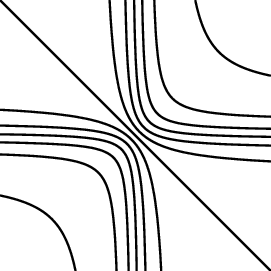}
		\caption{$\sum_{i=1}^2\arctan(\lambda_i) = const.$}
		\label{fig:arctan}
	\end{subfigure}
	\begin{subfigure}[b]{0.45\textwidth}
		\centering
		\includegraphics{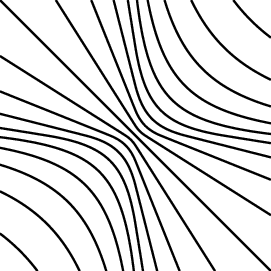}
		\caption{$\sum_{i=1}^2 a(\lambda_i) = const.$}
		\label{fig:alog}
	\end{subfigure}
	\caption{}\label{fig:a}
\end{figure}

\begin{proposition}\label{prop:alog_equation}
The comparison principle holds for the equation \eqref{eq:log_equation} in $\Omega$ whenever $f\colon\Omega\to\mR$ is locally Lipschitz.
\end{proposition}

\begin{proof}
The equation is elliptic with $\partial\Theta_+(x) = \partial\Theta_-(x)\neq\emptyset$ since $a$ is strictly increasing and onto $\mR$.

We take the standard approach.
Since $f(x)\leq F(X)\leq F(Z) < f(y)$, it follows that $|F(X)-f(x_0)|\leq Lt$
where $L\geq 0$ is the Lipschitz constant of $f$ restricted to $\oB_{t_0}(x_0)$. In particular,
\[-Lt+f(x_0) \leq F(X) = \sum_{i=1}^n a(\lambda_i(X))\leq n a(\lambda_n(X)),\]
which means that the largest eigenvalue of $X$ is bounded below. Say, $-M\leq \lambda_n(X)$.
Let $\tau>0$ be the solution to the equation $F(Z + \tau I) = f(y)$. Then $f(y) = F(Z + \tau I)\geq F(Z + \tau I) - F(X) + f(x)$ and
\[L|x-y|\geq f(y)-f(x)\geq F(Z + \tau I) - F(X).\]


We are now going to derive two different upper bounds for $\tau$. Consider first the case when also $\lambda_n(X)\leq M$. Then
\begin{align*}
2Lt &\geq F(Z + \tau I) - F(X)\\
	&\geq F(X + \tau I) - F(X)\\
    &\geq a(\lambda_n(X) + \tau) - a(\lambda_n(X))\\
    &= a'(\xi)\tau = \frac{\tau}{1+|\xi|}
\end{align*}
for some $\xi$ between $\lambda_n(X)$ and $\lambda_n(X)+\tau$.
In any case, $|\xi|\leq M + \tau$ and solving the above for $\tau$ then yields
\begin{equation}\label{eq:tau_bound_lambda_bounded}
\tau \leq 4(1+M)Lt
\end{equation}
when $t$ is so small so that $2Lt\leq 1/2$.

If $\lambda_n(X)> M\geq 0$, then
\begin{align*}
L|x-y| &\geq F(Z + \tau I) - F(X)\\
	  &\geq a(\lambda_n(Z) + \tau) - a(\lambda_n(X))\\
	  &= \ln(1 + \lambda_n(Z) + \tau) - \ln(1+\lambda_n(X))\\
	  &= \ln\left(\frac{1 + \lambda_n(Z) + \tau}{1+\lambda_n(X)}\right).
\end{align*}
That is,
\begin{align*}
\tau &\leq e^{L|x-y|}(1+\lambda_n(X)) - 1 - \lambda_n(Z)\\
       &= e^{L|x-y|} - 1 + e^{L|x-y|}\lambda_n(X) - \lambda_n(Z)\\
       &\leq e^{2Lt} - 1 + \frac{1}{\delta}\left(e^{L\frac{|x-y|}{2}} - 1\right)^2
\end{align*}
where the final inequality is due to Lemma \ref{lem:sm}. Together with \eqref{eq:tau_bound_lambda_bounded} this concludes the proof since $e^{2Lt}-1\to 0$ and since the last term is
\[\frac{t}{|x-y|^2}\left(L\frac{|x-y|}{2} + L^2\frac{|x-y|^2}{2!\cdot 2^2} + \cdots\right)^2\]
which also goes to zero as $t\to 0$.
\end{proof}

Similar calculations suggests that the result is true also when $f$ is locally Hölder continuous of order $\gamma\in (0,1)$ provided that $a$ is replaced with
\[a_\gamma(\lambda) := \int_0^\lambda\frac{\dd s}{1 + |s|^\gamma}.\]
It is not completely clear what the crucial properties of these equations are that makes the comparison principle hold. It is perhaps something like
\[F\left( T_\delta(X) + \tau I\right) - F(X) \geq \ell\left(\delta|X| + \frac{\tau}{1+|X|}\right)
\]
where $\ell$ is related to the modulus of continuity of $f$.


\section{The canonical operator $\oF$}\label{sec:oF}

Due to the idea introduced in \cite{MR1284912} -- which was further refined in \cite{MR2487853} and \cite{MR3677871} -- 
the study of viscosity sub- and supersolutions of elliptic equations
\[F(\cH u,x) = 0,\]
can be reduced to the study of the sub- and superlevel sets
\[\Theta_-(x) = \left\{X\in \Sn \;|\; F(X,x)\leq 0\right\},\quad \Theta_+(x) = \left\{X\in \Sn \;|\; F(X,x)\geq 0\right\}.\]
These are the only relevant objects. Krylov uses strict inequalities and assumes his elliptic sets to be open. \cite{MR2487853} and \cite{MR3677871} find it more natural to work with closed sets. When the comparison principle is the sole objective, it turns out that either assumption is superfluous.

When $\partial\Theta_-(x) = \partial\Theta_+(x)\neq\emptyset$ there is a canonical operator $\oF$ associated to the level sets that is \emph{consistent} with $F$. i.e., every (sub/super)solution of $F=0$ is also a (sub/super)solution of $\oF=0$. Namely, the signed distance function \eqref{eq:distop} from the common boundary $\Gamma(x) = \partial\Theta_\pm(x)$.
Its ellipticity was proved by Krylov (Theorem 3.2, \cite{MR1284912}). Various properties of similar constructions are derived in \cite{Harvey2020} where the equations also depend on $w$ and $\nabla w$.

\begin{proposition}\label{prop:oF}
Let $\Omega\subseteq\mR^n$ be open and assume that $\Theta_+$ is a proper positive elliptic map in $\Omega$. i.e.,
\[\emptyset\neq\Theta_+(x)\neq \Sn \qquad\text{and}\qquad\Theta_+(x) + \Sn_+  = \Theta_+(x)\qquad\forall x\in\Omega.\]
Define the function $\oF$ on $\Sn \times\Omega$ as
\begin{equation}
\oF(X,x) := -\inf\{t\in\mR\;|\; X+tI\in \Theta_+(x)\}.
\label{eq:acdo}
\end{equation}
Then $\oF$ is finite, elliptic, and 1-Lipschitz in $\Sn $, and has the nondegeneracy
\begin{equation}
\oF(X+\tau I,x) - \oF(X,x) = \tau
\label{eq:nondegen}
\end{equation}
for all $(X,x)\in \Sn \times\Omega$ and $\tau\in\mR$. 

Moreover, if $\Theta_+(x)$ is the superlevel set of an operator $F\colon \Sn \times\Omega\to\oR$, then every subsolution of $F(\cH w,x) = 0$ is also a subsolution to the equation $\oF(\cH w,x) = 0$ in $\Omega$. The opposite inclusion holds if $\Theta_+(x)$ is closed for all $x\in\Omega$.
\end{proposition}

\begin{corollary}\label{cor:oF}
If $F\colon \Sn \times\Omega\to\oR$ is elliptic at 0 and $\Gamma(x) := \partial\Theta_-(x) = \partial\Theta_+(x) \neq\emptyset$ for all $x\in\Omega$, we have the following alternative representations of the canonical operator \eqref{eq:acdo}.
\begin{enumerate}
	\item For each $(X,x)\in \Sn \times\Omega$, $\oF(X,x)$ is the unique number such that
	\[X - \oF(X,x)I\in\Gamma(x).\]
	\item
	\[\oF(X,x) = -\sup\{t\in\mR\;|\; X+tI\in \Theta_-(x)\}.\]
	\item
	\begin{equation}
	\oF(X,x) =
\begin{cases}
-\dist(X,\Gamma(x)) & \text{if $X\in\Theta_-(x)$},\\
\dist(X,\Gamma(x)) & \text{if $X\in\Theta_+(x)$}.
\end{cases}
	\label{eq:distop}
	\end{equation}
	\item \[\oF(X,x) = \dist(X,\Theta_-(x)) - \dist(X,\Theta_+(x)).\]
\end{enumerate}
Moreover, a (sub/super)solution of $F(\cH w,x) = 0$ is also a (sub/super)solution to the equation $\oF(\cH w,x) = 0$ in $\Omega$. If $\Theta_+(x)$ and $\Theta_-(x)$ is closed for all $x\in\Omega$, the two equations are equivalent.
\end{corollary}

Therefore, in order to check if a certain property (e.g. the comparison principle) holds for the sub- and supersolution of $F(\cH w,x) = 0$, it is sufficient to show that the property holds for the sub- and supersolution of $\oF(\cH w,x) = 0$. That is, for any $F\colon \Sn \times\Omega\to\oR$ elliptic at 0 with $\partial\Theta_+(x) = \partial\Theta_-(x)\neq\emptyset$, one can without loss of generality assume that $F$ is finite, elliptic, Lipschitz in $\Sn $, $F(X+\tau I,x)-F(X,x) = \tau$, and that the sub- and superlevel sets $\Theta_\mp(x)$ are closed. Furthermore, any existing good feature of $F$ -- like, for example, uniform ellipticity, regularity, convexity, positive homogenicity, or rotational invariance -- is often preserved, and sometimes enhanced, by $\oF$. This should be checked on a case to case basis.

\begin{example}
As evident from the proofs of Proposition \ref{prop:lin} and \ref{prop:lineig}, the canonical operators of $\tr(A(x)X)-f(x)$ and $\sum_{i=1}^n a_i(x)\lambda_i(X)-f(x)$ are $\frac{1}{\|A(x)\|_1}\left(\tr(A(x)X)-f(x)\right)$ and $\frac{1}{|\bfa(x)|_1}\left(\sum_{i=1}^n a_i(x)\lambda_i(X)-f(x)\right)$, respectively.
\end{example}

\begin{example}[The perturbed Monge-Ampère revisited]\label{ex:canonical_determinat}
Consider again
\[F(X,x) := \begin{cases}
\det(X+M(x)) - f(x),\qquad & \text{if $X+M(x)\geq 0$},\\
-\infty, &\text{otherwise,}
\end{cases}\]
with uniformly continuous $M\colon\Omega\to\Sn$ and $f\colon\Omega\to\mR$. When $n=2$, solving 
\[0 = F(X+\tau I,x) = \left(\lambda_1(X+M(x)) + \tau\right)\left(\lambda_2(X+M(x)) + \tau\right) - f(x)\]
for $\tau$ yields the canonical operator $\oF(X,x) = - \tau$. That is,
\[\oF(X,x) = \frac{1}{2}\tr(X+M(x)) - \frac{1}{2}\sqrt{\tr^2(X+M(x)) - 4(\det(X+M(x)) - f(x))}.\]
Alternatively,
\[\oF(X,x) = \frac{\lambda_1 + \lambda_2}{2} - \frac{1}{2}\sqrt{(\lambda_2-\lambda_1)^2 +4f(x)},\qquad \lambda_i := \lambda_i(X+M(x)).\]

Let $\Lambda := \lambda_2(X+M(x)) - \lambda_1(X+M(x)) \geq 0$ and $\Phi := 4f(x)$.
Since $M(y) \leq M(x) + \omega_M(|x-y|)I$,
we get
\[\lambda_2(X+M(y)) - \lambda_1(X+M(y)) \leq \Lambda + 2\omega_M(|x-y|).\]
By writing $\epsilon_M := 2\omega_M(|x-y|)$ and $\epsilon_f:=4\omega_f(|x-y|)$, it follows that
\begin{align*}
2\left|\oF(X,x) - \oF(X,y)\right|
	&\leq \epsilon_M + \sqrt{(\Lambda + \epsilon_M)^2 + \Phi + \epsilon_f} - \sqrt{\Lambda^2 + \Phi}\\
	&\leq \epsilon_M + \sqrt{\epsilon_f} + \sqrt{(\Lambda + \epsilon_M)^2 + \Phi} - \sqrt{\Lambda^2 + \Phi},
\end{align*}
which by standard analysis is bounded by $2\epsilon_M + \sqrt{\epsilon_f}$.
That is, $\left|\oF(X,x) - \oF(X,y)\right| \leq \omega_{\oF}(|x-y|)$ where $\omega_{\oF} := 2\omega_M + \sqrt{\omega_f}$. Thus, in contrast to $F$, the classical structure condition \eqref{eq:classical_structure_cond} holds for $\oF$ because the modulus of continuity $\omega_{\oF}$ is independent of $X$.
\end{example}

The properties of $\oF$ listed in the Proposition and Corollary are pointwise in $x$. We can therefore simplify, and prove the claims by considering autonomous equations. First, we settle some standard topological issues.

\begin{figure}[h]%
	\center
	\includegraphics{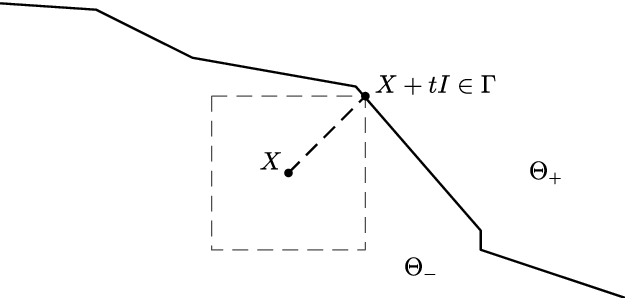}%
	\caption{A visualization of the equality of 2 and 3 in the Corollary. The ellipticity of the sublevel set $\Theta_-$ ensures that the largest ball $B_t(X)\subseteq\Theta_-$ (in the infinity norm) touches the boundary $\Gamma$ with its upper right corner.}%
	\label{fig:elliptic_set}%
\end{figure}

\begin{lemma}\label{lem:Thetaprop}
Let $F\colon \Sn \to\oR$ be elliptic at 0. Then the following hold.
\begin{enumerate}[(1)]
	\item $F(Z\mp A)\lessgtr 0$ for all $A>0$ and all $Z\in\partial\Theta_\pm$, respectively.
	\item If $Z\in\oT_\pm$, then $Z\pm A\in \Theta^\circ_\pm$ for all $A> 0$ and $Z\pm B\in \oT_\pm$ for all $B\geq 0$, respectively.
	\item The interior of the closure of the (sub/sup)level set equals its interior. That is,
	\[(\oT_\mp)^\circ = \Theta_\mp^\circ.\]
\end{enumerate}
\end{lemma}

\begin{proof}
It suffices to prove the claims for $\Theta_+$. The proofs concerning $\Theta_-$ are symmetric.
(1): Let $Z\in\partial\Theta_+, A>0$, and write $t:=\lambda_1(A)>0$. By definition of a boundary, there is a point $X_-$ in $B_t(Z)$ not in $\Theta_+$. i.e., $F(X_-)<0$. Since
\[t>|Z-X_-| \geq \lambda_n(Z-X_-),\]
we get $Z-A\leq Z-tI<X_-$, and the result follows from ellipticity at 0.

(2): Let $Z\in\oT_+$ and let $t := \lambda_1(A)>0$. Pick $Z_t\in \Theta_+\cap B_{t/2}(Z)$ and let $Y\in B_{t/2}(Z+A)$. Then
\[Y\geq Z+A-\frac{t}{2}I \geq Z_t -\frac{t}{2}I +A-\frac{t}{2}I \geq Z_t,\]
which, by ellipticity at 0 shows that $B_{t/2}(Z+A)\subseteq \Theta_+$. i.e., $Z+A\in\Theta_+^\circ$.
The last claim is immediate from this: Let $0<A\to B\geq 0$.

(3): The inclusion $\Theta_+^\circ\subseteq (\oT_+)^\circ$ is trivial. Let $Z\in(\oT_+)^\circ$, but suppose that $Z\notin\Theta_+^\circ$. This means that $Z\in\partial\Theta_+$ and that there is a ball $B_r(Z)\subseteq\oT_+$. In particular, the lower left octant of the ball $B_r(Z)\cap\{Z-A\;|\;A>0\}$ is in $\oT_+$ and must therefore contain a point $Z-A_0\in\Theta_+$. But $F(Z-A_0)<0$ by (1), which is a contradiction.

%
\end{proof}

\begin{proof}[Proof of the Proposition]
Since $\Theta_+$ is a proper subset, there exist matrices $Y\notin\Theta_+$ and $Z\in\Theta_+$. Let $X\in \Sn $. Then
\[X + sI \leq Y\quad \forall s \leq \lambda_1(Y-X),\qquad\text{and}\qquad X + \lambda_n(Z-X)I \geq Z,\]
which, by ellipticity implies that the set $\{t\;|\; X+tI\in\Theta_+\}$ is nonempty and bounded below. It is an interval since $X+tI\in\Theta_+$ implies $X+\tau I \in\Theta_+$ for all $t\leq \tau$. Its largest lower bound $-\oF(X)$ is then finite and, surely, $X - \oF(X)I\in\partial\Theta_+$.

We next prove ellipticity of $\oF$. Let $X\leq Y$. Then $Y - \oF(X)I \geq X - \oF(X)I\in\partial\Theta_+$, which implies $Y-\oF(X)I\in\oT_+$ by Lemma \ref{lem:Thetaprop} (2). Thus,
\[-\oF(Y) = \inf\{t\;|\; Y+tI\in \Theta_+\} = \inf\{t\;|\; Y+tI\in \oT_+\} \leq -\oF(X).\]

The nondegeneracy \eqref{eq:nondegen} is immediate:
\begin{align*}
-\oF(X+\tau I) &= \inf\{t\;|\; X+(t+\tau)I\in \Theta_+\}\\
	&= \inf\{s-\tau\;|\; X+sI\in \Theta_+\}\\
	&= \inf\{s\;|\; X+sI\in \Theta_+\} - \tau\\
	&= -F(X) -\tau.
\end{align*}

Let $X,Y\in \Sn $. Since $\lambda_1(Y)I\leq Y \leq \lambda_n(Y)I$ we get by ellipticity and \eqref{eq:nondegen} that
\begin{align*}
\oF(X+Y) &\leq \oF(X+\lambda_n(Y)I) = \oF(X) + \lambda_n(Y),\quad\text{and}\\
\oF(X+Y) &\geq \oF(X+\lambda_1(Y)I) = \oF(X) + \lambda_1(Y).
\end{align*}
Put together,
\[|\oF(X+Y)-\oF(X)| \leq \max\{-\lambda_1(Y), \lambda_n(Y)\} = |Y|,\]
and $\oF$ is 1-Lipschitz in $\Sn $.

The consistency follows simply because $\Theta_+(x)\subseteq\oT_+(x)$ and since $\oT_+(x)$ is the superlevel set of $\oF$. So, naturally, if $\Theta_+(x)$ is closed for all $x\in\Omega$, then $F=0$ and $\oF=0$ share the same set of subsolutions.
\end{proof}

\begin{proof}[Proof of the Corollary]
Part 1 was established in the proof of the Proposition and part 2 is the implication (vi) $\Rightarrow$ (iii) in Proposition \ref{thm:2}.

For part 3, assume first that $X\in\Theta_+$. Then $\oF(X)\geq 0$ and
\begin{align*}
\dist(X,\Gamma)
	&= \inf_{W\in\Gamma}|X-W|\\
	&\leq \left| X - (X-\oF(X)I)\right|\\
	&= \oF(X),
\end{align*}
by part 1. On the other hand, if $W\in\Theta_-$, then $W\geq X+\lambda_1(W-X)I\in\Theta_-$, so
\begin{align*}
\oF(X)
	&= -\sup\{t\;|\; X+tI\in\Theta_-\}\\
	&\leq - \sup\{\lambda_1(W-X)\;|\; W\in\Theta_-\}\\
	&= \inf\{-\lambda_1(W-X)\;|\; W\in\Theta_-\}\\
	&\leq \inf_{W\in\Theta_-}|X-W|\\
	&= \dist(X,\Gamma).	
\end{align*}
Similar computations apply when $X\in\Theta_-$.

Part 4 is immediate from 3, and the consistency and equivalency follows from duality and the Proposition.
\end{proof}

We have now established the essential properties of $\oF\colon \Sn \times\Omega\to\mR$ in $\Sn $. It remains to consider the behaviour of $\oF$ with respect to  $x\in\Omega$, but our main regularity result has to wait until the next Section.
Below we identify the property of an elliptic map that give rise to continuous canonical operators.

\begin{proposition}\label{prop:canonocal_cont_equiv}
Let $F\colon \Sn \times\Omega\to\oR$ be elliptic at 0 with $\partial\Theta_+(x) = \partial\Theta_-(x) \neq \emptyset$, and let $\oF$ denote the canonical operator \eqref{eq:acdo}. Then
$\oF$ is continuous in $\Sn \times\Omega$ if and only if
given $x_0\in\Omega$, $X_0\in \Gamma(x_0)$, and $\epsilon>0$ there is a $t>0$ such that
\begin{equation}\label{eq:contproperty}
X_0 \pm \epsilon I \in \Theta_\pm(x)\qquad\text{whenever}\qquad |x-x_0|\leq t.
\end{equation}
\end{proposition}

\begin{proof}
By the equi-continuity in $X$ it suffices to show that $x\mapsto \oF(X,x)$ is continuous for each fixed $X\in \Sn $. 
Let $x_0\in\Omega$ and $\epsilon>0$. By Corollary \ref{cor:oF}, $X_0 := X - \oF(X,x_0)I\in \Gamma(x_0)$. Choose $t>0$ so that $X_0\pm\epsilon I\in\Theta_\pm(x)$ whenever $|x-x_0|\leq t$. Then for such an $x$ also
\begin{align*}
\oF(X,x_0) - \oF(X,x)
&= -\oF\left(X - \oF(X,x_0)I,x \right)\\
&= \begin{cases}
\inf\left\{\tau\;|\; X_0 + \tau I \in\Theta_+(x) \right\} \leq \epsilon\\
\sup\left\{\tau\;|\; X_0 + \tau I \in\Theta_-(x) \right\} \geq -\epsilon.
\end{cases}
\end{align*}

For the other direction, assume that $\oF$ is continuous. Let $x_0\in\Omega$, $X_0\in \Gamma(x_0)$, and $\epsilon>0$. Then we can choose $t>0$ such that $\oF(X_0,x) = \oF(X_0,x) - \oF(X_0,x_0) < \epsilon$ whenever $|x-x_0|\leq t$. That is,
\[0 > \oF(X_0,x) - \epsilon = \oF(X_0 - \epsilon I,x)\]
and $X_0 - \epsilon I \in \Theta_-(x)$. Similarly, $\oF(X_0,x) > -\epsilon$ and $X_0 + \epsilon I \in \Theta_+(x)$.
\end{proof}

Note that the $+$ and $-$ in \eqref{eq:contproperty} correspond to lower- and upper semicontinuity of $x\mapsto\oF(X,x)$, respectively.

\section{The perturbation $T_\delta$}\label{sec:Tdelta}

For $\delta\in\mR$, let $\D_\delta = \D_\delta^n$ be the open and unbounded subset
\[\D_\delta := \{X\in \Sn \;|\; \delta X<I\}\]
of the space of symmetric $n\times n$ matrices.
Our goal is here to systematize the properties of the function
\[T_\delta\colon \D_\delta\to \D_{-\delta},\qquad T_\delta(X) := X(I-\delta X)^{-1},\]
which is an essential part of the uniqueness machinery of viscosity solutions.
It appears in the theory because, put simply, the \emph{sup-convolution}
\[\phi^\epsilon(x) := \sup_y\left\{\phi(y) - \frac{|x-y|^2}{2\epsilon}\right\}\]
of a quadratic function $\phi(x) = \frac{1}{2}x^T B x$ is $\phi^\epsilon(x) = \frac{1}{2}x^T T_\epsilon(B) x$ for small $\epsilon>0$. See \cite{MR1462699} exercise 11.2.

Note that $\D_0 = \Sn $ and that $T_0$ is the identity.
The group structure
\[T_{\alpha+\beta} = T_\alpha\circ T_\beta\]
(whenever defined) is easily proved and $T_\delta$ is thus a bijection with inverse $T_\delta^{-1} = T_{-\delta}$. This mapping of matrices is the \emph{lifting} of the strictly increasing \emph{scalar} function
\[T_\delta\colon \D_\delta\to \D_{-\delta},\qquad T_\delta(\lambda) := \frac{\lambda}{1-\delta \lambda}\]
where now $\D_\delta = \D_\delta^1 = \{\lambda\in\mR\;|\; \delta \lambda<1\}$. That is,
\[T_\delta(X) = \sum_{i=1}^n T_\delta(\lambda_i(X))\xi_i\xi_i^T\]
where $\xi_i$ is a choice of corresponding unit length eigenvectors of $X$. Since $\lambda\mapsto T_\delta(\lambda)$ is increasing, the order of the eigenvalues is preserved: $\lambda_i\circ T_\delta = T_\delta\circ\lambda_i$.
More generally, one can show that if $X,Y\in\D_\delta$ with $XY=0$, then $X+Y\in\D_\delta$ and $T_\delta(X+Y) = T_\delta(X) + T_\delta(Y)$. And if $Q$ is a $m\times n$ matrix such that $Q^TQ = I$, then $QXQ^T\in\D_\delta^m$ and
\[T_\delta\left(QXQ^T\right) = QT_\delta(X)Q^T.\]
In particular, $T_\delta(\lambda P) = T_\delta(\lambda)P$ for projections $P$ and scalars $\lambda\in\D_\delta^1$.

Furthermore, $B_{1/|\delta|} = \D_{-\delta}\cap\D_{\delta}\subseteq \D_{\pm\delta}$ (with the interpretation $B_{1/0}= \Sn $) and the perturbed level sets \eqref{eq:Tddef} can be written as 
\[\Theta_\mp^\delta(x) = T_{\mp\delta}\left(\Theta_\mp(x)\cap B_{1/\delta}\right),\qquad \delta\geq 0.\]
For $\delta\geq 0$, $0\leq c \leq 1$, and $X\in B_{c/\delta}$, we have
$1 + c> 1 - \delta\lambda_i(X) > 1 - c$,
which by the identity $T_\delta(X) = X + \delta X^2(I-\delta X)^{-1}$ yield the first order approximations
\begin{equation}\label{eq:T_firstorder_approx}
X + \frac{\delta}{1+c}X^2 \leq T_\delta(X) \leq X + \frac{\delta}{1-c}X^2.
\end{equation}
Of course, the upper bound is valid only for $c<1$.

As a side note, we mention that $T_\delta$ appears in the seemingly unrelated situation of Lemma 3.2 in \cite{MR4395603} as well: If $v\in C^2(\Omega)$ with $|\nabla v|\equiv 1$, then
\[\cH v\left(x - \delta\nabla v^T(x)\right) = T_\delta\left(\cH v(x)\right)\]
for all sufficiently small $|\delta|$.

\begin{proposition}[Monotonicity]\label{prop:Tmonotonicity}
Let $\delta\in\mR$ and $X,Y\in\D_\delta$. Then
\begin{enumerate}
\item 
\[X\leq Y\qquad\Rightarrow\qquad T_\delta(X) \leq T_\delta(Y).\]
\item\label{prop:T2} \[\frac{T_\alpha(X) - T_\beta(X)}{\alpha-\beta} = T_\alpha(X)T_\beta(X) \geq 0\] for all $\alpha\neq\beta$ between 0 and $\delta$. In particular,
\[\alpha\leq\beta\qquad\Rightarrow\qquad T_\alpha(X)\leq T_\beta(X).\]
\end{enumerate}
\end{proposition}

\begin{proof}
1: The assertion is that the scalar function $T_\delta$ is \emph{operator monotone} in $\D_\delta^1$. Assume first that $\delta > 0$ and let $X,Y\in\D_\delta$ with $X\leq Y$. Then $0 < I - \delta Y \leq I - \delta X$ and $(I-\delta X)^{-1} \leq (I-\delta Y)^{-1}$ (Proposition V.1.6 \cite{MR1477662}). By the identity $\delta X(I-\delta X)^{-1} = (I-\delta X)^{-1} - I$ it follows that
\[T_\delta(X) = \frac{1}{\delta}(I-\delta X)^{-1} - \frac{1}{\delta}I \leq \frac{1}{\delta}(I-\delta Y)^{-1} - \frac{1}{\delta}I = T_\delta(Y).\]
This is still true when $\delta<0$ because the inequalities change directions twice.

2: Since everything commutes, the factorization
\begin{align*}
T_\alpha(X) - T_\beta(X)
	&= X(I-\alpha X)^{-1} - X(I-\beta X)^{-1}\\
	&= (I-\alpha X)^{-1}(I-\beta X)^{-1}\Big[ X(I-\beta X) - X(I-\alpha X)\Big]\\
	&= (\alpha - \beta)(I-\alpha X)^{-1}(I-\beta X)^{-1} X^2\\
	&= (\alpha - \beta)T_\alpha(X)T_\beta(X)
\end{align*}
is valid. Also, a product of (semi)positive commuting matrices is (semi)positive.
\end{proof}

\begin{proposition}[Set equalities]\label{prop:Tset_equalities}
Let $\delta\in\mR$ and $m\geq 1$. Then 
\begin{enumerate}
\item\label{prop:Tm} $T_\delta(\D_{m\delta}) = \D_{-\delta}\cap\D_{(m-1)\delta}$. 
\item 
$T_\delta\left(B_{1/(m|\delta|)}\right) = \D_{-(m+1)\delta}\cap \D_{(m-1)\delta}$.
\end{enumerate}
\end{proposition}

From \ref{prop:Tm} we get in particular $T_\delta(\D_{\delta}) = \D_{-\delta}$ and $T_\delta(\D_{2\delta}) = B_{1/|\delta|}$.

\begin{figure}[h]
	\centering
	\begin{subfigure}[b]{0.45\textwidth}
		\centering
		\includegraphics{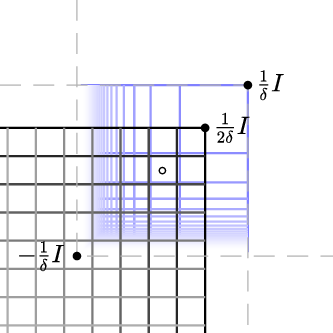}
		\caption{}
	\end{subfigure}
	\begin{subfigure}[b]{0.45\textwidth}
		\centering
		\includegraphics{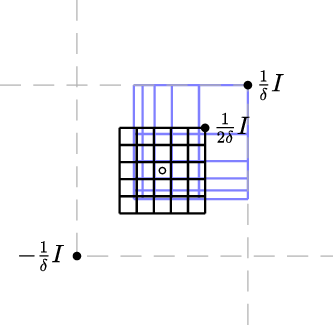}
		\caption{}
	\end{subfigure}
	\caption{The \color{blue} images \normalcolor in $T_\delta(\D_{2\delta}) = B_{1/\delta}$ and $T_\delta(B_{1/(2\delta)}) = \D_{-3\delta}\cap \D_{\delta}$ of the uniform grid covering the domains $\D_{2\delta}$ and $B_{1/(2\delta)}$, respectively. Here, $\delta>0$.}
\end{figure}

\begin{proof}
1: Let $X\in\D_{m\delta}$. Firstly, $m\delta X<I$, so $\delta X<I/m\leq I$ and $X\in\D_{\delta}$. Next, $I + \delta T_\delta(X) = (I-\delta X)^{-1}>0$, so $-\delta T_\delta(X) < I$ and $T_\delta(X)\in\D_{-\delta}$. Similarly,
$I - \delta(m-1)T_\delta(X) = (I-\delta X)^{-1}(I-\delta m X)>0$, which shows that $T_\delta(X)\in\D_{(m-1)\delta}$.
For the opposite inclusion $Y\in \D_{-\delta}\cap\D_{(m-1)\delta}$, one computes in the same way that $X := T_{-\delta}(Y)$ is in $\D_{m\delta}$.

2: For $X\in B_{1/(m|\delta|)}$ we have $I \pm(m\pm 1)\delta T_\delta(X) = (I-\delta X)^{-1}(I \pm m\delta X) > 0$, and thus $T_\delta(X)\in \D_{-(m+1)\delta}\cap\D_{(m-1)\delta}$. For $Y\in \D_{-(m+1)\delta}\cap\D_{(m-1)\delta}\subseteq \D_{-\delta}$ we get in the same way as above that $X := T_{-\delta}(Y) \in B_{1/(m|\delta|)}$.
\end{proof}

In analogy to the dual \eqref{eq:dual} of an elliptic set, the \emph{dual} of an operator $F\colon \Sn \times\Omega\to\oR$ is $\tF\colon \Sn \times\Omega\to\oR$ given by
\begin{equation}\label{eq:dualeq}
\widetilde{F}(X,x):= - F(-X,x).
\end{equation}
Likewise, we say that a \emph{condition} on $F$ is dual if it also holds for \eqref{eq:dualeq}. In connection with the comparison principle, it should be unsatisfying if a proposed sufficient condition is not dual. This is because the sub/supersolutions of $F=0$ are exactly the negative of the super/subsolution of the dual equation, and the condition is then certainly not necessary. More importantly, in our particular case it turns out that we need duality in order to get continuity of the canonical operator. As a step towards proving \eqref{eq:cond} is dual, we show that there is actually not necessary to take the supremum over all $Z\in \Theta_+^\delta(x) = T_\delta\left(\Theta_+(x)\cap B_{1/\delta}\right)$.

\begin{proposition}\label{prop:c_invariance}
	Let $F\colon \Sn \times\Omega\to\oR$ be elliptic at 0 with $\partial\Theta_+(x) = \partial\Theta_-(x)\neq\emptyset$ at all $x$ in $\Omega$. Then, for each $x_0\in\Omega$, the limit
	\begin{equation*}
	\lim_{t\to 0^+}\sup_{\substack{x,y\in B_t(x_0)\\ Z\in T_\delta\left(\Theta_+(x)\cap B_{c/\delta}\right)}}\dist\left(Z,\Theta_+(y)\right),\qquad \delta := \tfrac{|x-y|^2}{t},
	\end{equation*}
	exists and is independent of $0<c\leq 1$.
\end{proposition}

\begin{proof}
	Write the limit as $l(c)$.
	Since $B_{c/\delta}\subseteq B_{1/\delta}$ in $\Sn $, we have $l(c)\leq l(1)$. But
	since $Z\mapsto - \dist(Z,\Theta_+)$ is elliptic, $\delta\to T_\delta(X)$ is increasing, and $B_{cs}(x_0)\subseteq B_s(x_0)$ in $\Omega$ we also have
	\begin{align*}
	l(1) &= \lim_{t\to 0^+}\sup_{\substack{x,y\in B_t(x_0)\\ X\in\Theta_+(x)\cap B_{\frac{t}{|x-y|^2}}}}\dist\left(T_\frac{|x-y|^2}{t}(X),\Theta_+(y)\right)\\
	&\leq \lim_{t\to 0^+}\sup_{\substack{x,y\in B_t(x_0)\\ X\in\Theta_+(x)\cap B_{\frac{t}{|x-y|^2}}}}\dist\left(T_{c\frac{|x-y|^2}{t}}(X),\Theta_+(y)\right)\\
	&= \lim_{s\to 0^+}\sup_{\substack{x,y\in B_{cs}(x_0)\\ X\in\Theta_+(x)\cap B_{\frac{cs}{|x-y|^2}}}}\dist\left(T_{\frac{|x-y|^2}{s}}(X),\Theta_+(y)\right),\quad (s := t/c),\\
	&\leq \lim_{s\to 0^+}\sup_{\substack{x,y\in B_{s}(x_0)\\ X\in\Theta_+(x)\cap B_{\frac{cs}{|x-y|^2}}}}\dist\left(T_{\frac{|x-y|^2}{s}}(X),\Theta_+(y)\right)
	= l(c).
	\end{align*}
The limit exists and is non-negative because $\Theta_+(x)\cap B_{c/\delta}$ is nonempty for sufficiently small $\delta$, and with the same technique as above one can show that the expression inside the limit is an increasing function of $t$.
\end{proof}
Both ellipticity at 0 and the non-degeneracy $\partial\Theta_+ = \partial\Theta_-\neq\emptyset$ are obviously dual conditions. Since the suplevel set $\widetilde{\Theta}_+(x)$ of $\tF$ is $-\Theta_-(x)$ and $T_\delta(-Y) = -T_{-\delta}(Y)$ one can easily show that \eqref{eq:cond} for $\tF$ is \eqref{eq:newconddual} below.

\begin{proposition}[Duality of hypothesis]\label{prop:duality_of_hyp}
	Let $F\colon \Sn \times\Omega\to\oR$ be elliptic at 0 with $\partial\Theta_+(x_0) = \partial\Theta_-(x_0)\neq\emptyset$ at $x_0\in\Omega$. Then
	\begin{align}
	\lim_{t\to 0^+}&\sup_{\substack{x,y\in B_t(x_0)\\ X\in\Theta_+(x)\cap B_{1/\delta}}}\dist\left(T_\delta(X),\Theta_+(y)\right) = 0\label{eq:newcond}\\
	&\text{if and only if}\notag\\
	\lim_{t\to 0^+}&\sup_{\substack{x,y\in B_t(x_0)\\ Y\in\Theta_-(y)\cap B_{1/\delta}}}\dist\left(T_{-\delta}(Y),\Theta_-(x)\right) = 0.\label{eq:newconddual}
	\end{align}
\end{proposition}

\begin{proof}
It suffices to show one direction. Suppose \eqref{eq:newconddual} is not true. By Proposition \ref{prop:c_invariance} (and by duality of its assumptions) we can replace $B_{1/\delta}$ with $B_{1/(3\delta)}$. Then there is a sequence $t_k\searrow 0$, points $x_k,y_k\in B_{t_k}(x_0)$, and $Y_k \in\Theta_-(y_k)\cap B_{1/(3\delta_k)}$ such that, for all $k$,
	\[\dist\left(T_{-\delta_k}(Y_k),\Theta_-(x_k)\right) > \tau\]
	for some positive $\tau$. Define $Z_k := T_{-\delta_k}(Y_k)$ and $X_k := Z_k - \tau I$. Then
	\[X_k > Z_k - \dist\left(Z_k,\Theta_-(x_k)\right)I \in\Gamma(x_k),\]
	which means that $X_k \in\Theta_+(x_k)$. Moreover,
	\[Z_k \in  \D_{-(m+1)(-\delta_k)}\cap \D_{(m-1)(-\delta_k)} = \D_{4\delta_k}\cap \D_{-2\delta_k}\]
	by Proposition \ref{prop:Tset_equalities}. Consider $k$ so large so that $\delta_k\tau < 1/4$ and write $\epsilon_k := \frac{\delta_k}{1-\delta_k\tau} > \delta_k$. Now,
	\[\delta_k X_k = \delta_k (Z_k - \tau I) < \left(\frac{1}{4} - \delta_k\tau\right)I < \left(1 - \delta_k\tau\right)I,\]
	which shows that $X_k\in\D_{\epsilon_k}\subseteq \D_{\delta_k}$. Note that $1+\epsilon_k\tau = \frac{1}{1-\delta_k\tau}$ and $\frac{\epsilon_k}{1+\epsilon_k\tau} = \delta_k$. Thus,
	\begin{align*}
	T_{\epsilon_k}(X_k)
	&= T_{\epsilon_k}\left(Z_k - \tau I\right)\\
	&= (Z_k - \tau I)\left( I - \epsilon_k(Z_k - \tau I)\right)^{-1}\\
	&= (Z_k - \tau I)\left( (1+\epsilon_k\tau)\left(I - \frac{\epsilon_k}{1+\epsilon_k\tau}Z_k\right)\right)^{-1}\\
	&= (1-\delta_k\tau)(Z_k - \tau I)\left( I - \delta_k Z_k\right)^{-1}\\
	&= T_{\delta_k}(Z_k) - \tau\Big[\delta_k T_{\delta_k}(Z_k) + (1-\delta_k\tau)\left( I - \delta_k Z_k\right)^{-1}\Big].
	\end{align*}
	Next, since $0 < I-\delta_k Z_k < (1+\frac{1}{2})I = \frac{3}{2}I$, we have $(I-\delta_k Z_k)^{-1}>\frac{2}{3}I$ and
	\begin{align*}
	T_{\delta_k}(X_k)
	&\leq T_{\epsilon_k}(X_k)\\
	&= Y_k - \tau\Big[\delta_k Y_k + (1-\delta_k\tau)\left( I - \delta_k Z_k\right)^{-1}\Big]\\
	&\leq Y_k - \tau\Big[-\frac{1}{3} + (1-\delta_k\tau)\frac{2}{3}\Big]I\\
	&\leq Y_k - \frac{\tau}{6}I.
	\end{align*}
	Finally, $X_k\in B_{1/\delta_k}$ as $X\in\D_{\delta_k}$ and $X_k > - (\frac{1}{2\delta_k} + \tau)I > - \frac{1}{\delta_k}I$. Thus,
\[\dist\left(T_{\delta_k}(X_k),\Theta_+(y_k)\right)
\geq \dist\left(Y_k - \frac{\tau}{6}I,\Theta_+(y_k)\right)\\
\geq \frac{\tau}{6}\]
is a contradiction of \eqref{eq:newcond}.
\end{proof}

\begin{proposition}\label{prop:contT}
Under the hypothesis of Theorem \ref{thm}, the canonical operator
is continuous in $\Sn \times\Omega$.
\end{proposition}

\begin{proof}
We shall show lower semicontinuity of $x\mapsto\oF(X,x)$ under the condition \eqref{eq:newcond}. Then, $\overline{\tF}(X,x)$ is l.s.c. as well by Proposition \ref{prop:duality_of_hyp}. A straight forward calculation yields $\overline{\tF} = \widetilde{\oF}$,
which means that $x\mapsto \oF(-X,x)=-\widetilde{\oF}(X,x)$ is also upper semicontinuous, and thus continuous for all $X\in \Sn $.

	Let $x_0\in\Omega$, $X_0\in\Gamma(x_0)$ and $\epsilon>0$ be given. Choose a small $t>0$ so that $t|X_0|<\frac{1}{4}$, $\frac{4}{3}t|X_0|^2<\frac{\epsilon}{2}$, and
	\[\sup_{\substack{x,y\in B_t(x_0)\\ Z\in\Theta_+^{\delta}(y)}} - \oF(Z,x) \leq\sup_{\substack{x,y\in B_t(x_0)\\ Z\in\Theta_+^{\delta}(y)}} \dist\left( Z,\Theta_+(x)\right) < \epsilon/2.\]
Now,
	\begin{align*}
	-\frac{\epsilon}{2}
	&< \inf_{\substack{x,y\in B_t(x_0)\\ X\in\Theta_+(y)\cap B_\frac{t}{|x-y|^2}}} \oF\left( T_\frac{|x-y|^2}{t}(X),x\right)\\
	&\leq \inf_{\substack{x\in B_t(x_0)\\ X\in\Theta_+(x_0)\cap B_\frac{t}{|x-x_0|^2}}} \oF\left( T_\frac{|x-x_0|^2}{t}(X),x\right)\\
	&\leq \inf_{\substack{x\in B_t(x_0)\\ X\in\Theta_+(x_0)\cap B_{1/t}}} \oF\left( T_t(X),x\right)
	\end{align*}
	since $\frac{|x-x_0|^2}{t} < t$ and thus $B_\frac{t}{|x-x_0|^2}\supseteq B_{1/t}$ and $T_\frac{|x-x_0|^2}{t}(X)\leq T_t(X)$. Surely, $X_0\in \oT_+(x_0)\cap B_{1/t}$ and
	\[T_t(X_0) \leq X_0 + \frac{t}{1-c}X_0^2 = X_0 + \frac{4t}{3}X_0^2 \leq X_0 + \frac{\epsilon}{2}I\]
	by \eqref{eq:T_firstorder_approx}, ($c=1/4$). That is,
	\[0 < \frac{\epsilon}{2} + \oF\left( X_0 + \frac{\epsilon}{2}I,x\right) = \oF\left( X_0 + \epsilon I,x\right)\]
	and $X_0 + \epsilon I\in \Theta_+(x)$ for all $x\in B_t(x_0)$. We conclude by Proposition \ref{prop:canonocal_cont_equiv}.
\end{proof}

The following connection between $T_\delta$ and the matrix inequality (3.10) in \cite{MR1118699} is probably well known, but we include the proof since we have not been able to find a suitable reference for the exact result below.

\begin{proposition}\label{prop:matrixineq}
	Let $X,Y\in \Sn $ and let $\alpha>0$. Then
	\begin{equation}
	\begin{bmatrix}
	X & 0\\ 0 & -Y
	\end{bmatrix} \leq \alpha\begin{bmatrix}
	I & -I\\ -I & I
	\end{bmatrix}
	\label{eq:struc}
	\end{equation}
	\[\text{if and only if}\]
	\[\epsilon X < I\qquad\text{and}\qquad X(I-\epsilon X)^{-1}\leq Y\quad\forall\epsilon\in[0,1/\alpha).\]
\end{proposition}

The ``only if''-part will be needed in the proof of the Theorem. 
The ``if''-part was used in many of the Propositions in Section \ref{sec:examples} via the implication

\begin{corollary}\label{cor:matrixineq}
	If $\delta > 0$ and $I - \delta X > 0$, then 
	\[\begin{bmatrix}
	X & 0\\ 0 & - X(I-\delta X)^{-1}
	\end{bmatrix} \leq \frac{1}{\delta}\begin{bmatrix}
	I & -I\\ -I & I
	\end{bmatrix}.\]
\end{corollary}

In fact, $T_\epsilon(X)\leq T_\delta(X) =: Y$ whenever $\epsilon\in[0,\delta)$ by Proposition \ref{prop:Tmonotonicity}.

\begin{proof}[Proof of Proposition]
	$\underline{\Rightarrow}:$
	The claim is verified for $\epsilon = 0$ by multiplying \eqref{eq:struc} from the left and right with vectors in $\mR^{2n}$ on the form $\bigl[\begin{smallmatrix}
	\xi \\ \xi
	\end{smallmatrix}\bigr]$. Let $\epsilon\in(0,1/\alpha)$. Multiplying with $\bigl[\begin{smallmatrix}
	\xi \\ 0
	\end{smallmatrix}\bigr]$ shows that $X\leq \alpha I < \frac{1}{\epsilon}I$, and $I-\epsilon X$ is therefore positive definite and thus invertible. Next, for $\eta\in\mR^n$ choose $\xi\in\mR^n$ to be
	\[\xi = (I-\epsilon X)^{-1}\eta.\]
	The right-hand side matrix in \eqref{eq:struc} is positive semidefinite, so the inequality continues to hold with $\alpha$ replaced with $1/\epsilon >\alpha$. That is,
	\begin{align*}
	0 &\leq
	\begin{bmatrix}
	\xi^T & \eta^T
	\end{bmatrix}\left( \frac{1}{\epsilon}\begin{bmatrix}
	I & -I\\ -I & I
	\end{bmatrix} - \begin{bmatrix}
	X & 0\\ 0 & -Y
	\end{bmatrix}\right)\begin{bmatrix}
	\xi\\ \eta
	\end{bmatrix}\\
	&= \frac{1}{\epsilon}\Big(|\xi|^2 + |\eta|^2 - 2\eta^T\xi\Big) - \xi^TX\xi + \eta^T Y\eta\\
	&= \frac{1}{\epsilon}\Big( \eta^T\left(I+\epsilon Y\right)\eta + \xi^T\left(I-\epsilon X\right)\xi - 2\eta^T\xi \Big)\\
	&= \frac{1}{\epsilon}\eta^T\Big( I + \epsilon Y - \left(I-\epsilon X\right)^{-1} \Big)\eta
	\end{align*}
	for all $\eta\in\mR^n$. It follows that
	\[Y \geq \frac{1}{\epsilon}\left(I-\epsilon X\right)^{-1} - \frac{1}{\epsilon}I = \frac{1}{\epsilon}\left(I-\epsilon X\right)^{-1}\big(I - (I-\epsilon X)\big) = X\left(I-\epsilon X\right)^{-1}.\]
	
	$\underline{\Leftarrow}:$
	Let $\bigl[\begin{smallmatrix}
	\xi \\ \eta
	\end{smallmatrix}\bigr]\in\mR^{2n}$. Since $I-\epsilon X$ is positive, the expression
	\begin{align*}
	\begin{bmatrix}
	\xi^T & \eta^T
	\end{bmatrix}&\left(\begin{bmatrix}
	X & 0\\ 0 & -Y
	\end{bmatrix} - \frac{1}{\epsilon}\begin{bmatrix}
	I & -I\\ -I & I
	\end{bmatrix}\right)\begin{bmatrix}
	\xi\\ \eta
	\end{bmatrix}\\
	&= \frac{1}{\epsilon}\left( -\eta^T\left(I+\epsilon Y\right)\eta - \xi^T\left(I-\epsilon X\right)\xi + 2\eta^T\xi \right),
	\end{align*}
	considered as a quadratic function of $\xi$, is maximized at $\xi = (I-\epsilon X)^{-1}\eta$ with value
	\[\frac{1}{\epsilon}\eta^T\left( \left(I-\epsilon X\right)^{-1} - I - \epsilon Y \right)\eta = \eta^T\left( X\left(I-\epsilon X\right)^{-1}  - Y \right)\eta \leq 0.\]
	This confirms \eqref{eq:struc} since it holds for all numbers $1/\epsilon > \alpha$. 
\end{proof}

\section{Proof of the Theorem}\label{sec:proof}

The standard tool when proving comparison principles is the \emph{Theorem of Sums} or \emph{Ishii's Lemma}. It produces points in the \emph{sub- or superjet closures} $\overline{J}^{2,\mp}w_i(\hat{x}_i)\subseteq \mR^n\times \Sn $ at a critical point $(\hat{x}_1,\dots, \hat{x}_N)$ for a sum of semicontinuous functions $w_i(x_i)$ in $\Omega\times\cdots\times\Omega$. The result is a corner-stone in the viscosity theory and relies on the use of sup/inf-convolutions and, ultimately, on Alexandrov's theorem which states that a convex function is twice differentiable almost everywhere. We shall not go into the details here, but rather restate the result in a simple form that suffices for our needs. For the definitions and proof, we refer to Section 2 and 3 of \cite{MR1118699}.

\begin{lemma}[Theorem of Sums]\label{lem:ishii}
	Let $\Omega\subseteq\mR^n$ be open and bounded.
	Suppose $v\in USC(\oO)$, $u\in LSC(\oO)$, and assume that $(x_k,y_k)\in \Omega\times \Omega$ is a maximum point of
	\[v(x) - u(y) - \frac{k}{2}\lvert x-y\rvert^2,\qquad k= 1,2,\dots,\]
	in $\oO\times\oO$. Then there are matrices $X_k,Y_k\in \Sn $ such that
	\[\big(k(x_k-y_k),X_k\big) \in\overline{J}^{2,+}v(x_k)\qquad\text{and}\qquad \big(k(x_k-y_k),Y_k\big) \in\overline{J}^{2,-}u(y_k),\]
	and where
	\[-3k \begin{bmatrix}
	I & 0\\ 0 & I
	\end{bmatrix}\leq
	\begin{bmatrix}
	X_k & 0\\ 0 & -Y_k
	\end{bmatrix} \leq
	3k \begin{bmatrix}
	I & -I\\ -I & I
	\end{bmatrix}\]
	in $\mathcal{S}^{2n}$.
\end{lemma}

Viscosity solutions can be defined from the point of view of sub- and superjets. In particular, if $v$ is a, say, subsolution to an elliptic equation $F(\cH w, \nabla w,w,x) = 0$ in $\Omega$, and 
$(p,X)$ is in the superjet $J^{2,+}v(x_0)$ for some $x_0\in\Omega$, then $F(X,p,v(x_0),x_0)\geq 0$.
However, in order to come to the same conclusion when $(p,X)$ is only in the superjet \emph{closure} $\overline{J}^{2,+}v(x_0)$, we need $F$ to be at least upper semicontinuous. Lower semicontinuity is needed for the subjets, and Proposition \ref{prop:contT} is thus expedient.

The following Lemma is the reason we only need to consider points $x$ and $y$ such that $x,y\to x_0\in\Omega$ in Theorem \ref{thm}.


\begin{lemma}\label{lem:touch}
Let $w\in USC(\oO)$ in an open and bounded $\Omega\subseteq\mR^n$ with $w|_{\partial\Omega}\leq 0$. If $w>0$ somewhere in $\Omega$, there is a quadratic
\[\phi(x) = a + b^Tx - \frac{\tau}{2}|x|^2,\qquad a\in\mR,\,b\in\mR^n,\,\tau > 0,\]
touching $w$ strictly from above at some point $x_0\in\Omega$. That is,
$\phi(x_0) = w(x_0)$ and $\phi>w$ in $\oO\setminus\{x_0\}$.
\end{lemma}

\begin{proof}
Recall that an upper semicontinuous function obtains its maximum on compact sets.
Set $\beta := \max_{\oO} w > 0$ and let $R>0$ be so large so that $\oO\subseteq B_R(0)$. Next, let $\tau := \beta/R^2$ and put
\begin{equation}
c := \max_{y\in\oO}\left(w(y) + \tau|y|^2\right).
\label{eq:c}
\end{equation}
Consider the function
\[\psi(x) := c - \tau|x|^2\]
which obviously is greater or equal to $w$ in $\oO$. Also, $c\geq\beta$, so the maximum in \eqref{eq:c} must be obtained at an interior point $x_0\in\Omega$ since, for $\eta\in\partial\Omega$ we have
\[w(\eta) + \tau|\eta|^2 \leq \frac{\beta}{R^2}|\eta|^2 < \beta.\]
Thus, $\psi(x_0) = w(x_0)$, and we may take $\phi(x) := \psi(x) + \frac{\tau}{2}|x-x_0|^2$.
\end{proof}

We are now ready to assemble the proof of Theorem \ref{thm}. Since our equations are independent of the gradient, we slightly abuse the notation and consider the semi jets as subsets of $\Sn $ instead of $\mR^n\times \Sn $.

Let $u$ and $v$ be super- and subsolutions of the equation $F(\cH w,x) = 0$. By Corollary \ref{cor:oF}, we may replace $F$ with its canonical representative $\oF$.
Assume $v|_{\partial\Omega}\leq u|_{\partial\Omega}$ and suppose to the contrary that $v>u$ somewhere in $\Omega$. As $v-u\in USC(\oO)$, Lemma \ref{lem:touch} provides a test function $\phi$ touching $v-u$ strictly from above at some point $x_0\in\Omega$.
Setting
\[\hat{u}(x) := u(x) + \phi(x)\]
yields $v(x_0) - \hat{u}(x_0) = 0$ and $v - \hat{u} < 0$ in $\oO\setminus\{x_0\}$.
The next computations are standard. For $k=1,2,\dots$, let $(x_k,y_k)\in\oO\times\oO$ be the maximum point of
\[v(x) - \hat{u}(y) -\frac{k}{2}\lvert x-y\rvert^2.\]
By compactness we may assume, by taking a subsequence if necessary, that $(x_k,y_k)$ converges to some point $(x^*,y^*)$ in $\oO\times\oO$ as $k\to\infty$.
Since
\[v(x_k) - \hat{u}(y_k) -\frac{k}{2}\lvert x_k-y_k\rvert^2 \geq  v(x_0) - \hat{u}(x_0) = 0\]
and $v-\hat{u}$ is bounded above by semicontinuity, it follows that $x^*=y^*$. Moreover, by taking the limsup of $0 \leq \frac{k}{2}\lvert x_k-y_k\rvert^2 \leq v(x_k) - \hat{u}(y_k)$
we have
\[\lim_{k\to\infty}k\lvert x_k-y_k\rvert^2 = 0\]
and $v(x^*) - \hat{u}(x^*) = 0$. Thus $y^* = x^* = x_0$ being the only touching point of $v$ and $\hat{u}$ in $\oO$. In particular, $(x_k,y_k)$ is eventually in $\Omega\times \Omega$.

By Lemma \ref{lem:ishii} there are points $X_k$ and $\hat{Y}_k$ in the semi jet closures of $v(x_k)$ and $\hat{u}(y_k)$ such that
\[-3k\begin{bmatrix}
	I & 0\\ 0 & I
\end{bmatrix}\leq \begin{bmatrix}
	X_k & 0\\ 0 & -\hat{Y}_k
\end{bmatrix} \leq 3k  \begin{bmatrix}
	I & -I\\ -I & I
\end{bmatrix}.\]
Since $\hat{u} = u+\phi$, we may write $\hat{Y}_k = Y_k + \mathcal{H}\phi(y_k) = Y_k - \tau I$ where $\tau >0$ is the constant from Lemma \ref{lem:touch}, and where $Y_k$ is in the subjet closure of $u(y_k)$. By virtue of Proposition \ref{prop:contT},
\[\oF(X_k,x_k) \geq 0\qquad\text{and}\qquad 0 \geq \oF(Y_k,y_k) = \oF(\hat{Y}_k + \tau I,y_k) = \oF(\hat{Y}_k,y_k) + \tau.\]
Set
\[t_k := 3k|x_k-y_k|^2 + |x_k-x_0| + |y_k-x_0| + 1/k.\]
Then $x_k,y_k\in B_{t_k}(x_0)$,  $0\leq \delta_k := \frac{|x_k-y_k|^2}{t_k} < \frac{1}{3k}$, 
and Proposition \ref{prop:matrixineq} yields $Z_k := X_k\left(I-\delta_k X_k\right)^{-1}\leq \hat{Y}_k$.
Moreover, $\delta_k|X_k| < \frac{3k}{3k} = 1$, and thus $Z_k\in \Theta_+^{\delta_k}(x_k)$.

Since $t_k\searrow 0$ as $k\to\infty$, our condition \eqref{eq:cond} creates the contradiction
\begin{align*}
\tau &\leq \lim_{k\to\infty}-\oF(\hat{Y}_k,y_k)\\
     &\leq \lim_{k\to\infty}-\oF(Z_k,y_k)\\
	 &\leq \lim_{k\to\infty} \dist\left( Z_k, \Theta_+(y_k) \right)\\
	 &\leq \lim_{k\to\infty}\sup_{\substack{x,y\in B_{t_k}(x_0)\\ Z\in \Theta_+^{\delta_k}(x)}} \dist\left( Z, \Theta_+(y) \right) = 0,
\end{align*}
which proves the Theorem.

\section{A linear equation without superposition- and comparison principles}\label{sec:lin}

When it comes to the comparison principle in linear elliptic equations
\[\tr(A(x)\cH w) = 0,\]
an immediate first observation is that $A$ cannot be allowed to vanish at some point in the domain.\footnote{At least not in our setting with a classical definition of viscosity sub- and supersolutions. There is, however, a theory for $L^p$-viscosity solutions where this can be allowed since the ingredients then are interpreted only outside sets of measure zero. See e.g. \cite{MR1376656}.} Because if $A(x_0) = 0$, then the (sub)solution $v\equiv 0$ and the lower semicontinuous supersolution
\[u(x) =
\begin{cases}
0,\qquad &\text{if $x\neq x_0$,}\\
-1,&\text{if $x = x_0$.}
\end{cases}\]
constitute a counterexample.
On the other hand, one should remember that ``$\leq$'' is only a partial ordering in $\Sn $, so $0\leq A(x)\neq 0$ does not imply $0<A(x)$. That is, the equation does not necessarily have to be strictly elliptic.
Indeed, by Proposition \ref{prop:lin} it is sufficient if $A$ is on the form
\[A(x) = \sum_{i=1}^m \bfq_i(x)\bfq_i^T(x)\]
for some locally Lipschitz vector fields $\bfq_i\colon\Omega\to\mR^n\setminus\{0\}$, $i = 1,\dots,m$.
Our following example shows that comparison may fail if the Lipschitz condition is not met.

Consider the equation
\begin{equation}
x^{2/3}w_{xx} - 2(xy)^{1/3}w_{xy} + y^{2/3}w_{yy} = 0
\label{eq:countereq}
\end{equation}
in regions $(0,0)\notin\Omega$ of the plane.
It can be written as $\tr\big(A(x,y)\cH w\big) = 0$
where $A\colon \mR^2\to \mathcal{S}^2_+ $ is given by
\[A(x,y) := \begin{bmatrix}
	x^{2/3} & -(xy)^{1/3}\\ -(xy)^{1/3} & y^{2/3}
\end{bmatrix}.\]
By defining
\[\bfq(x,y) := \left[
	x^{1/3},\, - y^{1/3}\right]^T\]
we see that $A(x,y) = \bfq(x,y)\bfq^T(x,y)$ and $A$ has constant rank one away from the origin. Still, the equation \eqref{eq:countereq} does not satisfy the hypothesis of Proposition \ref{prop:lin} in subsets $(0,0)\notin\Omega\subseteq\mR^2$ because $A/\lVert A\rVert_1 = \bfq\bfq^T/|\bfq|^2$
and $\bfq/|\bfq|$ is not Lipschitz at the coordinate axes.

The reason for this choice of equation is that $\bfq^T$ is, modulo a factor $4/3$, the gradient of Aronsson's function
\[u(x,y) := x^{4/3} - y^{4/3},\]
which is $\infty$-harmonic in the viscosity sense in $\mR^2$. Therefore, $u$ is also a solution to the linear equation \eqref{eq:countereq}. To see this, assume that $\phi$ touches $u$ from, say, below at $(x_0,y_0)\in\mR^2$ then $\nabla\phi(x_0,y_0) = \frac{4}{3}\bfq^T(x_0,y_0)$ as $u$ is $C^1$ and
\[\tr\big(A\cH \phi\big) = \bfq^T\cH\phi\bfq = \frac{9}{16}\nabla\phi\cH\phi\nabla\phi^T = \frac{9}{16}\Delta_\infty\phi\leq 0\qquad\text{at $(x_0,y_0)$.}\]

We claim that the piecewise linear function
\[v(x,y) := |x| - |y|\]
is a solution to \eqref{eq:countereq} as well. We only have to check on the axes. There are no test functions touching from above on $\{x=0\}$ and there are no touching from below on $\{y=0\}$. Suppose therefore that $\phi$ touches $v$ from above at $(x_0,0)$, $x_0\neq 0$. Then $\phi_{yy}(x_0,0)$ can be arbitrarily negative, but
\[\phi_{xx}(x_0,0) = \lim_{\epsilon\to 0}\frac{\phi(x_0-\epsilon,0) - 2\phi(x_0,0) + \phi(x_0+\epsilon,0)}{\epsilon^2}\geq 0.\]
Since $A(x_0,0) = x_0^{2/3}e_1e_1^T$, it follows that
\[\tr\big(A(x_0,0)\cH \phi(x_0,0)\big) = x_0^{2/3}\phi_{xx}(x_0,0)\geq 0.\]
Similarly, $\tr\big(A\cH \phi\big)\leq 0$ at a touching point $(0,y_0)$ for test functions $\phi\leq v$.

We remark that $v$ is not $\infty$-harmonic since the gradient of the test functions does not necessarily align with the coordinate axes at the touching points $(x_0,0)$ and $(0,y_0)$.

\begin{figure}[h]
	\centering
	\begin{subfigure}[b]{0.45\textwidth}
		\centering
		\includegraphics{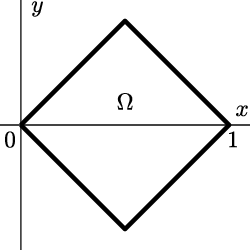}
		\label{fig:domain}
	\end{subfigure}
	\begin{subfigure}[b]{0.45\textwidth}
		\centering
		\includegraphics{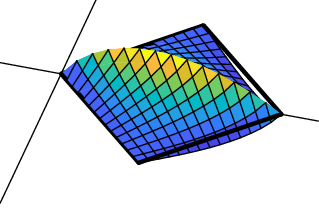}
		\label{fig:fun}
	\end{subfigure}
	\caption{The difference $v-u$ has an interior maximum in $\Omega$.}
	\label{fig:comparison}
\end{figure}

Now consider the two solutions $u$ and $v$ of \eqref{eq:countereq} in the diamond-shaped domain $\Omega$ given by $|y| < \min\{x,1-x\}$. It is bounded by the four line segments
\begin{align*}
\ell_1^\pm &:= \big\{(x,y)\;|\; y=\pm x,\,0\leq x\leq 1/2\big\}\quad\text{and}\\
\ell_2^\pm &:= \big\{(x,y)\;|\;y=\pm(1-x),\,1/2\leq x\leq 1\big\}.
\end{align*}
Both $u$ and $v$ are zero on $\ell_1^\pm$, and on $\ell_2^\pm$ we have 
\[u(x,\pm(1-x)) \geq 2x - 1 = v(x,\pm(1-x)),\qquad 1/2\leq x\leq 1.\]
This can be checked by using that $u(1/2,\pm 1/2) = 0$ and $u(1,0) = 1$, and that $x\mapsto u(x,\pm(1-x))$ is concave.
Therefore $v|_{\partial\Omega} \leq u|_{\partial\Omega}$, and the comparison principle is violated as, on the interior line $(x,0)$, $0<x<1$,
\[u(x,0) = x^{4/3} < x = v(x,0).\]

Note that the vanishing of $A$ at the boundary point $(0,0)$ is not the issue. We can cut away the leftmost corner of
$\Omega$, add a small constant to $u$, and come to the same conclusion.
Also note that $\tr\big(A(x,y)(-\tau I)\big) = -(x^{2/3} + y^{2/3})\tau$ is negative in $\Omega$ for all $\tau>0$, and even uniformly when the corner is gone. Together with Lemma \ref{lem:touch}, this property shows that the maximum principle \emph{is} valid for the equation.
Therefore, the function $v-u$ is not a subsolution and there is no superposition principle in the equation \eqref{eq:countereq} as well.

\paragraph{Acknowledgments:}
Supported by the Academy of Finland (grant SA13316965) and Aalto University.
We thank the Reviewer for the comments on the earlier version of the manuscript, and for the suggestions on how to improve it.

\bibliographystyle{alpha}
\bibliography{/Users/karlkb/Documents/references.bib}

\begin{thebibliography}{CHLP20}

\bibitem[Bee93]{MR1269778}
Gerald Beer.
\newblock {\em Topologies on closed and closed convex sets}, volume 268 of {\em
  Mathematics and its Applications}.
\newblock Kluwer Academic Publishers Group, Dordrecht, 1993.

\bibitem[Bha97]{MR1477662}
Rajendra Bhatia.
\newblock {\em Matrix analysis}, volume 169 of {\em Graduate Texts in
  Mathematics}.
\newblock Springer-Verlag, New York, 1997.

\bibitem[BM06]{MR2246004}
Martino Bardi and Paola Mannucci.
\newblock On the {D}irichlet problem for non-totally degenerate fully nonlinear
  elliptic equations.
\newblock {\em Commun. Pure Appl. Anal.}, 5(4):709--731, 2006.

\bibitem[Bru21]{MR4395603}
Karl~K. Brustad.
\newblock Segre's theorem. {A}n analytic proof of a result in differential
  geometry.
\newblock {\em Asian J. Math.}, 25(3):321--340, 2021.

\bibitem[Bru22]{arxiv.2206.09373}
Karl~K. Brustad.
\newblock Counterexamples to the comparison principle in the special
  {L}agrangian potential equation.
\newblock {\em arXiv:2206.09373}, 2022.

\bibitem[CCKS96]{MR1376656}
L.~Caffarelli, M.~G. Crandall, M.~Kocan, and A.~Swiech.
\newblock On viscosity solutions of fully nonlinear equations with measurable
  ingredients.
\newblock {\em Comm. Pure Appl. Math.}, 49(4):365--397, 1996.

\bibitem[CHLP20]{Harvey2020}
Marco Cirant, F.~Reese Harvey, H.~Blaine Lawson, and Kevin~R. Payne.
\newblock Comparison principles by monotonicity and duality for constant
  coefficient nonlinear potential theory and pdes.
\newblock {\em arXiv:2009.01611v1}, 2020.

\bibitem[CIL92]{MR1118699}
Michael~G. Crandall, Hitoshi Ishii, and Pierre-Louis Lions.
\newblock User's guide to viscosity solutions of second order partial
  differential equations.
\newblock {\em Bull. Amer. Math. Soc. (N.S.)}, 27(1):1--67, 1992.

\bibitem[CP17]{MR3677871}
Marco Cirant and Kevin~R. Payne.
\newblock On viscosity solutions to the {D}irichlet problem for elliptic
  branches of inhomogeneous fully nonlinear equations.
\newblock {\em Publ. Mat.}, 61(2):529--575, 2017.

\bibitem[CP21]{MR4147574}
Marco Cirant and Kevin~R. Payne.
\newblock Comparison principles for viscosity solutions of elliptic branches of
  fully nonlinear equations independent of the gradient.
\newblock {\em Math. Eng.}, 3(4):Paper No. 030, 45, 2021.

\bibitem[Cra97]{MR1462699}
Michael~G. Crandall.
\newblock Viscosity solutions: a primer.
\newblock In {\em Viscosity solutions and applications ({M}ontecatini {T}erme,
  1995)}, volume 1660 of {\em Lecture Notes in Math.}, pages 1--43. Springer,
  Berlin, 1997.

\bibitem[HL82]{MR666108}
Reese Harvey and H.~Blaine Lawson, Jr.
\newblock Calibrated geometries.
\newblock {\em Acta Math.}, 148:47--157, 1982.

\bibitem[HL09]{MR2487853}
F.~Reese Harvey and H.~Blaine Lawson, Jr.
\newblock Dirichlet duality and the nonlinear {D}irichlet problem.
\newblock {\em Comm. Pure Appl. Math.}, 62(3):396--443, 2009.

\bibitem[HL19]{MR4286829}
F.~Reese Harvey and H.~Blaine Lawson, Jr.
\newblock The inhomogeneous {D}irichlet problem for natural operators on
  manifolds.
\newblock {\em Ann. Inst. Fourier (Grenoble)}, 69(7):3017--3064, 2019.

\bibitem[HL21]{MR4179860}
F.~Reese Harvey and H.~Blaine Lawson, Jr.
\newblock Pseudoconvexity for the special {L}agrangian potential equation.
\newblock {\em Calc. Var. Partial Differential Equations}, 60(1):Paper No. 6,
  37, 2021.

\bibitem[IL90]{MR1031377}
H.~Ishii and P.-L. Lions.
\newblock Viscosity solutions of fully nonlinear second-order elliptic partial
  differential equations.
\newblock {\em J. Differential Equations}, 83(1):26--78, 1990.

\bibitem[Koi04]{MR2084272}
Shigeaki Koike.
\newblock {\em A beginner's guide to the theory of viscosity solutions},
  volume~13 of {\em MSJ Memoirs}.
\newblock Mathematical Society of Japan, Tokyo, 2004.

\bibitem[Kry95]{MR1284912}
N.~V. Krylov.
\newblock On the general notion of fully nonlinear second-order elliptic
  equations.
\newblock {\em Trans. Amer. Math. Soc.}, 347(3):857--895, 1995.

\bibitem[LE05]{MR2142457}
Yousong Luo and Andrew Eberhard.
\newblock Comparison principles for viscosity solutions of elliptic equations
  via fuzzy sum rule.
\newblock {\em J. Math. Anal. Appl.}, 307(2):736--752, 2005.

\end{thebibliography}


\end{document}